\documentclass{fundam2}
\usepackage{url} 
\usepackage[ruled,lined]{algorithm2e}
\usepackage{graphicx}

\usepackage{amssymb}
\usepackage{amsfonts}
\usepackage{yfonts}
\usepackage{pifont}
\usepackage{bold-extra}
 \usepackage{mathptmx}
\usepackage{diagrams}

\renewcommand{\QED}{\hfill \ding{113}}

\begin{document}
\setcounter{page}{1}
\issue{\copyright \,\textsf{2019} \textsc{Saeed Salehi \& Mohammadsaleh Zarza}}

\title{First-Order Continuous Induction, \\ and a Logical Study of  Real Closed  Fields$^\bigstar$}

\address{\sl \textsc{Saeed Salehi}, Research Institute for Fundamental Sciences \textup{(}RIFS\textup{)}, University of Tabriz,  P.O.Box~51666--16471, Bahman 29$^\text{th}$ Boulevard, Tabriz, IRAN.
\\
\textup{$\bigstar$ This is a part of the Ph.D. thesis of the second author written under the supervision of the first author at the University of Tabriz, {\sc Iran}.}}

\author{Saeed Salehi\\
Research Institute for Fundamental Sciences \textup{(RIFS)},
University of Tabriz, \\ P.O.Box~51666--16471,   Bahman 29$^\text{th}$ Boulevard, Tabriz, IRAN\\
http://saeedsalehi.ir/ \; salehipour{@}tabrizu.ac.ir
\and
Mohammadsaleh  Zarza\\
Department of Mathematics,
University of Tabriz, \\ P.O.Box~51666--16471,   Bahman 29$^\text{th}$ Boulevard, Tabriz, IRAN\\
msz1982{@}gmail.com
} \maketitle

\runninghead{\textup{S. Salehi  \&         M.  Zarza}}
        {First-Order Continuous Induction, and a Logical Study of  Real Closed  Fields}

\begin{abstract}
Over the last century, the principle of   ``induction on the continuum'' has been studied by different authors in different formats. All of these different readings are equivalent to one of the three versions that we isolate in this paper. We also formalize those three forms  (of ``continuous induction'') in  first-order logic and prove that two of them are equivalent and sufficiently strong to completely axiomatize the first-order theory of the real closed (ordered) fields. We show that the third weaker form of  continuous induction is  equivalent with the Archimedean property.
We study some equivalent axiomatizations for the theory of real closed fields and propose a first-order scheme of the fundamental theorem of algebra as an alternative axiomatization for this theory (over the theory of ordered fields).
\end{abstract}

\begin{keywords}
First-Order Logic;  Complete Theories;  Axiomatizing the Field of Real Numbers;  Continuous Induction;  Real Closed Fields.

\textbf{2010 AMS MSC}:
 03B25,  
 03C35, 
 03C10, 
 12L05. 
\end{keywords}

\section{Introduction}\label{sec:intro}
Real Analysis
is more than a haphazard accumulation of facts about the ordered field of  real numbers $\mathbb{R}$.
 Indeed, Real Analysis is a systematic study of $\mathbb{R}$ (and functions on $\mathbb{R}$, etc.). The most usual systematic way of studying mathematical objects is via axiomatizations. Some axiomatic systems are just definitions; such as the axioms of Group Theory. Indeed, examples of groups abound in mathematics and other scientific fields.
However, some axiomatizations are much deeper than definitions; one such example is the axiom system of  Complete Ordered Fields. An {\em ordered   field} is {\em complete} when it contains the supremum (least upper bound)
of every nonempty and bounded
subset of itself (let us note that this notion of completeness is formalized in second-order logic).
There are lots of ordered fields in Mathematics, but only one of them, up to isomorphism, is complete (by Dedekind's Theorem; see e.g. \cite{spivak}).
The
assumption of the existence of a complete ordered field is not a trivial one; though many textbooks on Mathematical Real Analysis start off with the axioms of complete ordered fields, and take $\mathbb{R}$ as a (indeed, as {\em the}) model of this theory.
Of course, this is not the only way to do real analysis; by the arithmetization of analysis, we can construct $\mathbb{R}$ from $\mathbb{Q}$, and $\mathbb{Q}$ from $\mathbb{Z}$, and finally $\mathbb{Z}$ from $\mathbb{N}$.
It should be pointed out that we can also reverse this foundational set-up by starting
with the assumption that $\mathbb{R}$ is a complete ordered field; then we can construct $\mathbb{Q}$ as the smallest subfield of
$\mathbb{R}$, $\mathbb{Z}$ as the smallest sub-ring of $\mathbb{R}$ that contains $1$, and $\mathbb{N}$ as the non-negative elements of $\mathbb{Z}$.

Apart from these philosophical and foundational concerns, in this paper we are interested in the properties of $\mathbb{R}$ as a complete ordered  field.
Indeed, there are several different axiomatizations  for  the (second-order) theory of complete ordered fields, over the theory of Ordered Fields $\texttt{\textbf{OF}}$: (1)~the existence of supremum (the least upper bound) for every nonempty and bounded subset; (2)~the existence of infimum (the greatest lower bound) for every nonempty and bounded subset; and (3)~the nonexistence of cuts (with gaps), i.e., a partition into two disjoint subsets in such a way that every element of the first set is smaller than all the elements of the second set, and the first set has no greatest element and the second set has no smallest element.
Interestingly, there are many different equivalent statements for the completeness axiom;
indeed, as many as 72
of them are listed
 in \cite{deteisman}.

As shown by Tarski, the first-order theory of $\mathbb{R}$, i.e., the collection of first-order sentences in the language of ordered fields that hold in the ordered field of real numbers, is axiomatizable by a computable
set of axioms known to modern algebraists as $\texttt{\textbf{RCF}}$ (real closed fields---see Definition~\ref{def:rcf}). Therefore,
$\texttt{\textbf{RCF}}$ is a complete theory, i.e., any first order sentence formulated in the language of ordered fields is
either provable or refutable  in $\texttt{\textbf{RCF}}$. This is in sharp contrast to the semiring $\mathbb{N}$ of natural numbers, the
ring $\mathbb{Z}$ of integers, and the field $\mathbb{Q}$ of rationals, none of whose first-order theories can be axiomatized
by a computable set of axioms, thanks to the work of Kurt G\"odel (for $\mathbb{N}$ and $\mathbb{Z}$) and Julia Robinson (for
$\mathbb{Q}$). Thus, the formalization of the fundamentals of Real Analysis naturally leads to central concepts
in Modern Algebra. As we shall see, formalizing various completeness axioms for the ordered field of
real numbers may (in some cases) give rise to various axiomatizations of the theory of real closed fields.

In Section~\ref{sec:ci} we isolate three schemes of the principle of continuous induction, since all the different formats of this principle that have appeared in the literature over the last century are equivalent to one of these three versions.
We will show that two of them are equivalent to each other, and to the completeness principle over $\texttt{\textbf{OF}}$; while the third one is weaker. We formalize these three schemes in first-order logic and will compare their strength with each other.
In Section~\ref{sec:fo} we study some first-order formalizations of the completeness axiom (of  ordered fields) and will see their equivalence with one another by first-order proofs. We will show that the weak principle of continuous induction is equivalent to the Archimedean property (of ordered abelian groups), and the first-order formalization of the (two equivalent) strong principle(s) of continuous induction can axiomatize the theory of real closed fields (over $\texttt{\textbf{OF}}$).
In Section~\ref{sec:rcf} we study the theory of real closed fields more deeply and introduce a first-order scheme of the Fundamental Theorem of Algebra as an alternative axiomatization of this theory over $\texttt{\textbf{OF}}$. In Section~\ref{sec:conc}, we summarize the new and old results of the paper and propose a set of open problems for future investigations. In the Appendix we present a slightly modified proof of Kreisel and Krivine \cite[Chapter~4, Theorem~7]{kk} for Tarski's Theorem on the
 completeness (and the decidability of a first-order axiomatization) of the theory of real closed fields.

\section{Continuous Induction, Formalized in First-Order Logic}\label{sec:ci}
``Continuous Induction'',  ``Induction over the Continuum'',  ``Real Induction'',  ``Non-Discrete Induction'', or the like, are some terms used by authors for referring  to some statements about the continuum  $\mathbb{R}$. These statements are  as strong as the Completeness Axiom of $\mathbb{R}$ and a motivation for their introduction into the literature of mathematics is the easy and sometimes unified ways they provide for proving some basic theorems of Mathematical Analysis. Here, we do not intend to give a thorough history of the subject or list all of the relevant literature in the References. For a ``telegraphic history'' we refer the reader to  \cite{kalantari}; and for an introduction to the subject we refer to \cite{clark} and the references therein. The earliest use of continuous induction is perhaps  the  paper \cite{chao} dating back to 1919. Below, we will formalize it in first-order logic. We will also formalize the formulations of  continuous induction presented in \cite{clark,jing,kalantari} and \cite{hath}; and later will compare their strength with each other. Let us finally note that
 \cite{mashkouri} is the only textbook (in Persian/Farsi) in which we could find some mention of  continuous induction (referring   to \cite{jing}).

The principle of continuous induction introduced in \cite{chao} (see also \cite{clark}) is equivalent to the following statement:

\begin{definition}[$\texttt{\textbf{CI}}_1$]\label{def:ind''}

\noindent
For any $S\subseteq\mathbb{R}$, if
\begin{enumerate}
\item for some $a\!\in\!\mathbb{R}$ we have $]\!-\infty,a]\subseteq S$, and
\item there exists some $\epsilon\!\!>\!\!0$ such that for all  $x\!\in\!\mathbb{R}$ if $x\!\in\!S$, then $[x,x\!+\!\epsilon]\subseteq S$,
\end{enumerate}
then $S\!=\!\mathbb{R}$.
\hfill\ding{71}\end{definition}

Its formalization in first-order logic (where the language contains $\{<,0,+\}$) is:

\begin{definition}[$\texttt{\textbf{DCI}}_1$]
\label{def:foind''}

\noindent
Let $\texttt{\textbf{DCI}}_1$ (definable continuous induction) be the following first-order scheme
$$\exists x\,\forall y\!\!\leqslant\!\!x\,\varphi(y)\,\wedge\,  \exists\epsilon\!\!>\!\!0\,\forall x\big(\varphi(x)\rightarrow\forall y\,[x\!\!\leqslant\!\!y\!\!\leqslant\!\!x\!+\!\epsilon
\rightarrow\varphi(y)]\big)\longrightarrow\forall x\,\varphi(x)$$
where $\varphi$ is an arbitrary formula.
\hfill\ding{71}\end{definition}

Continuous induction in \cite{kalantari} (see also \cite{jing} and \cite{clark}) is equivalent to the following:

\begin{definition}[$\texttt{\textbf{CI}}$]\label{def:ind}

\noindent
For any $S\subseteq\mathbb{R}$, if
\begin{enumerate}
\item for some $a\!\in\!\mathbb{R}$ we have $]\!-\infty,a[\,\subseteq S$, and
\item for any $x\!\in\!\mathbb{R}$ if $]\!-\infty,x[\,\subseteq S$, then there exists some  $\epsilon\!\!>\!\!0$ such that $]\!-\infty,x\!+\!\epsilon[\,\subseteq S$,
\end{enumerate}
then $S\!=\!\mathbb{R}$.
\hfill\ding{71}\end{definition}

Let us note that this form of real induction can be formulated by using  $<$ only:

\noindent For any $S\subseteq\mathbb{R}$, if
\begin{enumerate}
\item for some $a\!\in\!\mathbb{R}$ we have $]\!-\infty,a[\,\subseteq S$, and
\item for any $x\!\in\!\mathbb{R}$ if $]\!-\infty,x[\,\subseteq S$, then there exists some  $y\!\!>\!\!x$ such that $]\!-\infty,y[\,\subseteq S$,
\end{enumerate}
then $S\!=\!\mathbb{R}$.

The first-order formalization of this is:

\begin{definition}[$\texttt{\textbf{DCI}}$]
\label{def:foind}

\noindent
Let $\texttt{\textbf{DCI}}$ be the first-order scheme
$$\exists x\,\forall y\!\!<\!\!x\,\varphi(y)\,\wedge\,  \forall x\big[\forall y\!\!<\!\!x\,\varphi(y)\rightarrow\exists z\!\!>\!\!x\,\forall y\!\!<\!\!z\,\varphi(y)\big]\longrightarrow\forall x\,\varphi(x)$$
for an arbitrary formula  $\varphi$.
\hfill\ding{71}\end{definition}

When the first-order language contains $0$ and $+$ also, then this can be formalized as

$$\exists x\,\forall y\!\!<\!\!x\,\varphi(y)\,\wedge\,  \forall x\big[\forall y\!\!<\!\!x\,\varphi(y)\rightarrow\exists \epsilon\!\!>\!\!0\,\forall y\!\!<\!\!x\!+\!\epsilon\,\varphi(y)\big]\longrightarrow
\forall x\,\varphi(x).$$

Note that in $\texttt{\textbf{DCI}}_1$ there exists a fixed $\epsilon\!\!>\!\!0$ such that for all $x$ the second assumption holds, but in $\texttt{\textbf{DCI}}$ for any $x$ there exists some $\epsilon_x\!\!>\!\!0$ (which depends on $x$) such that the second assumption holds.

Finally, there exists a third version of continuous induction which appears in \cite{hath}:

\begin{definition}[$\texttt{\textbf{CI}}_2$]\label{def:ind'}

\noindent
For any $S\subseteq\mathbb{R}$, if
\begin{enumerate}
\item for some $a\!\in\!\mathbb{R}$ we have $]\!-\infty,a]\subseteq S$, and
\item for any $x\!\in\!\mathbb{R}$ if $(-\infty,x]\subseteq S$, then there exists some  $y\!\!>\!\!x$ such that $]\!-\infty,y[\,\subseteq S$, and
\item for any $x\!\in\!\mathbb{R}$ if $]\!-\infty,x[\,\subseteq S$, then $x\!\in\!S$,
\end{enumerate}
then $S\!=\!\mathbb{R}$.
\hfill\ding{71}\end{definition}

Its formalization in the first-order languages that contain $<$ is as follows:

\begin{definition}[$\texttt{\textbf{DCI}}_2$]
\label{def:foind'}

\noindent
Let $\texttt{\textbf{DCI}}_2$ denote  the first-order scheme
$$\exists x\,\forall y\!\!\leqslant\!\!x\,\varphi(y) \wedge   \forall x\big[\forall y\!\!\leqslant\!\!x\,\varphi(y)\rightarrow\exists z\!\!>\!\!x\,\forall y\!\!<\!\!z\,\varphi(y)\big]\wedge\forall x[\forall y\!\!<\!\!x\,\varphi(y)\rightarrow\varphi(x)]
\longrightarrow\forall x\,\varphi(x)$$
where  $\varphi$ is an arbitrary formula.
\hfill\ding{71}\end{definition}

Now, we compare the strength of these three schemes  with each other.

\begin{theorem}[$\texttt{\textbf{DCI}}
\Longleftrightarrow\texttt{\textbf{DCI}}_2$]
\label{th:indind'}

\noindent
In any  linear order   $\langle D;<\rangle$, the scheme \textup{$\texttt{\textbf{DCI}}$} holds, if and only if \textup{$\texttt{\textbf{DCI}}_2$} holds.
\end{theorem}
\begin{proof}
The proof of $\texttt{\textbf{DCI}}
\Longrightarrow\texttt{\textbf{DCI}}_2$ is rather easy (and so it is left to the reader).

\noindent
For $\texttt{\textbf{DCI}}_2
\Longrightarrow\texttt{\textbf{DCI}}$, suppose that $\texttt{\textbf{DCI}}_2$ holds in a linearly ordered set $\langle D;<\rangle$, and  assume  that for a formula $\varphi(x)$ and some $a\!\in\!D$ we have

\begin{itemize}\itemindent=1.75em
\item[(i)]~$\forall y\!\!<\!\!a\,\varphi(y)$,  and
\item[(ii)]~$\forall x\big[\forall y\!\!<\!\!x\,\varphi(y)\rightarrow\exists z\!\!>\!\!x\,\forall y\!\!<\!\!z\,\varphi(y)\big]$.
\end{itemize}

\noindent
Now, we show that the following relations hold:
\begin{itemize}\itemindent=1.75em
\item[(1)] $\exists x\,\forall y\!\!\leqslant\!\!x\,\varphi(y)$,
\item[(2)] $\forall x\big[\forall y\!\!\leqslant\!\!x\,\varphi(y)\rightarrow\exists z\!\!>\!\!x\,\forall y\!\!<\!\!z\,\varphi(y)\big]$, and
\item[(3)] $\forall x[\forall y\!\!<\!\!x\,\varphi(y)\rightarrow\varphi(x)]$.
\end{itemize}

\noindent
This will show (by using $\texttt{\textbf{DCI}}_2$) that $\forall x\varphi(x)$ holds.

 The relation (2) follows straightforwardly from (ii). For (3) fix some $d\in D$ and assume that $\forall y\!\!<\!\!d\,\varphi(y)$ holds. Then by (ii) there exists some $d'\!\!>\!\!d$ such that
$\forall y\!\!<\!\!d'\,\varphi(y)$ holds too. Whence, $\varphi(d)$ holds as well. The relation (1) holds for $x\!\!=\!\!a$ for exactly the same reason.
 \end{proof}

It follows from the proof of Theorem~\ref{th:indind'} that  $\texttt{\textbf{CI}}$ and $\texttt{\textbf{CI}}_2$ are equivalent with each other, in any  linear order. Now we show that $\texttt{\textbf{CI}}_1$ is strictly weaker than $\texttt{\textbf{CI}}$ (and $\texttt{\textbf{CI}}_2$).

\begin{theorem}[$\texttt{\textbf{DCI}}
\not\Longleftrightarrow\texttt{\textbf{DCI}}_1$]
\label{th:indind''}

\noindent
In any  ordered divisible abelian group, if \textup{$\texttt{\textbf{DCI}}$} holds, then \textup{$\texttt{\textbf{DCI}}_1$} holds too. But not vice versa: \textup{$\texttt{\textbf{DCI}}_1$} holds in the rational numbers $\mathbb{Q}$ but \textup{$\texttt{\textbf{DCI}}$} does not.
\end{theorem}
\begin{proof}
Suppose that $\texttt{\textbf{DCI}}$ holds in such a structure $\langle D;+,0,<,\mathcal{L}\rangle$. For proving $\texttt{\textbf{DCI}}_1$ suppose that for some formula $\varphi$ and  some $\epsilon,a\!\in\!D$ with $\epsilon\!\!>\!\!0$ we have
\begin{itemize}\itemindent=1.75em
\item[(i)] $\forall y\!\!\leqslant\!\!a\,\varphi(y)$ and

\item[(ii)] $\forall x\big(\varphi(x)\rightarrow\forall y\,[x\!\!\leqslant\!\!y\!\!\leqslant\!\!x\!+\!\epsilon
\rightarrow\varphi(y)]\big)$.
\end{itemize}

\noindent
Then  (1)~$\forall y\!\!<\!\!a\,\varphi(y)$ holds and we show that
(2)~$\forall x\big[\forall y\!\!<\!\!x\,\varphi(y)\rightarrow\exists z\!\!>\!\!x\,\forall y\!\!<\!\!z\,\varphi(y)\big]$ holds, as follows: for some fixed $b\!\in\!D$ assume that $\forall y\!\!<\!\!b\,\varphi(y)$ holds. Then we have $\varphi(b\!-\!\frac{\epsilon}{2})$ and so (ii) implies that $\forall y\,[b\!-\!\frac{\epsilon}{2}\!\!\leqslant\!\!y
\!\!\leqslant\!\!b\!+\!\frac{\epsilon}{2}\rightarrow\varphi(y)]$. Whence, $\forall y\!\!<\!\!b\!+\!\frac{\epsilon}{2}\,\varphi(y)$ holds, which proves (2). So,  $\texttt{\textbf{DCI}}$  implies $\forall x\,\varphi(x)$ from (1) and (2).

Now, we show that $\texttt{\textbf{DCI}}_1$ holds in  the (ordered field of) rational numbers  $\mathbb{Q}$ but $\texttt{\textbf{DCI}}$ does not.
For proving $\mathbb{Q}\vDash\texttt{\textbf{DCI}}_1$ suppose that for some formula $\varphi$ and  some numbers $\epsilon,r\!\in\!\mathbb{Q}$ with $\epsilon\!\!>\!\!0$ we have
(i)~$\forall y\!\!\leqslant\!\!r\,\varphi(y)$ and
(ii)~$\forall x\big(\varphi(x)\rightarrow\forall y\,[x\!\!\leqslant\!\!y\!\!\leqslant\!\!x\!+\!\epsilon
\rightarrow\varphi(y)]\big)$. We show that $\forall x\,\varphi(x)$ holds in  $\mathbb{Q}$; if not, then  $\mathcal{A}\!=\!\{q\!\in\!\mathbb{Q}\mid \neg\varphi(q)\}$ is nonempty and bounded below by $r$. Let $\alpha\!=\!\inf \mathcal{A}(\!\in\!\mathbb{R})$. So,  there exists some $s\!\in\!\mathcal{A}$ with $\alpha\!\!<\!\!s\!\!<\!\!\alpha\!+\!\frac{\epsilon}{2}$, and  there exists some $t\!\in\!\mathbb{Q}$ with $\alpha\!-\!\frac{\epsilon}{2}\!\!<\!\!t\!\!<\!\!\alpha$.
By $t\!\!<\!\!\alpha\!=\!\inf \mathcal{A}$ we   have $t\!\not\in\!\mathcal{A}$ and so $\varphi(t)$   holds. Then by (ii) we should have   $\varphi(y)$ for all $y\!\in\![t,t\!+\!\epsilon]$, and in particular $\varphi(s)$, since we already have $t\!\!<\!\!\alpha\!\!<\!\!s\!\!<\!\!\alpha\!+\!\frac{\epsilon}{2}\!\!<\!\!t\!+\!\epsilon$;
which is a contradiction with $s\!\in\!\mathcal{A}$.
We now show that  $\texttt{\textbf{CI}}$ does not hold  in $\mathbb{Q}$. Let $\varphi(x)\!=\![x\!\!<\!\!0\vee x^2\!\!<\!\!2]$. Obviously, (1)~$\forall y\!\!<\!\!0\,\varphi(y)$ holds and we show that
(2)~$\forall x\big[\forall y\!\!<\!\!x\,\varphi(y)\rightarrow\exists \epsilon\!\!>\!\!0\,\forall y\!\!<\!\!x\!+\!\epsilon\,\varphi(y)\big]$ holds as well. But clearly $\forall x\,\varphi(x)$ does not hold since $\varphi(2)$ is not true; this will show that  $\texttt{\textbf{DCI}}$ is not true in $\mathbb{Q}$. For showing (2) fix an  $r\!\in\!\mathbb{Q}$ and assume that $\forall y\!\!<\!\!r\,\varphi(y)$.
We show $\exists \epsilon\!\!>\!\!0\,\forall y\!\!<\!\!r\!+\!\epsilon\,\varphi(y)$ by distinguishing the following three cases:
\begin{itemize}
\item[(I)] If $r\!\!<\!\!0$, then for $\epsilon\!=\!\!-\!\frac{r}{2}$ we have  $\forall y\!\!<\!\!r\!+\!\epsilon\,\varphi(y)$.

\item[(II)] If $r\!=\!0$, then for $\epsilon\!=\!1$ we have $\forall y\!\!<\!\!r\!+\!\epsilon\,\varphi(y)$.

\item[(III)] If $r\!\!>\!\!0$, then we have $r^2\!\!<\!\!2$, since $r^2\!=\!2$ is impossible (by $\sqrt{2}\!\not\in\!\mathbb{Q}$) and if $r^2\!\!>\!\!2$, then we cannot have $\forall y\!\!<\!\!r\,\varphi(y)$ because for
$y=\frac{r^2+2}{2r}$ we have $0\!\!<\!\!y\!\!<\!\!r$ and $y^2\!-\!2\!=\!(\frac{r^2-2}{2r})^2\!\!>\!\!0$.
Let $\epsilon\!=\!\min\{1,\frac{2-r^2}{2r+1}\}$; then  $0\!\!<\!\!\epsilon\!\!\leqslant\!\!1$. For showing $\forall y\!\!<\!\!r\!+\!\epsilon\,\varphi(y)$, take some $y$ with  $0\!\!\leqslant\!\!y\!\!<\!\!r\!+\!\epsilon$. Then 
\newline\centerline{$y^2\!-\!2\!\!<\!\!(r\!+\!\epsilon)^2\!-\!2
\!=\!(r^2\!-\!2)\!+\!\epsilon(2r\!+\!\epsilon)
\!\!\leqslant\!\!(r^2\!-\!2)\!+\!\epsilon(2r\!+\!1)
\!\!\leqslant\!\!(r^2\!-\!2)\!+\!(2\!-\!r^2)\!=\!0.$}
So, $y^2\!\!<\!\!2$, thus $\varphi(y)$ holds for any $y\!\!<\!\!r\!+\!\epsilon$.
 \end{itemize}
 \vspace{-1.75em}
 \end{proof}

Therefore, $\texttt{\textbf{CI}}_1$ cannot be regarded as a genuine ``continuous induction''; even though it is a kind of ``non-discrete induction'' since it does not hold in $\mathbb{Z}$ or $\mathbb{N}$. But since it holds in $\mathbb{Q}$ then it is not really an ``induction on the continuum''.
In the next section we will  see the real reason  for regarding  $\texttt{\textbf{CI}}$ (and also its equivalent $\texttt{\textbf{CI}}_2$)  as a genuine ``continuous induction'', and not regarding $\texttt{\textbf{CI}}_1$ as a true ``induction on the continuum''.

\section{First-Order Completeness Axioms for the Field of Real Numbers}\label{sec:fo}
In Mathematical Analysis, $\mathbb{R}$ is usually introduced as a Complete Ordered Field; indeed by a theorem of Dedekind  there exists  only one complete ordered field up to isomorphism. In most of the  textbooks on Mathematical Analysis, the existence of a complete ordered field is assumed as an axiom;
 in some other textbooks, a complete ordered field is constructed from $\mathbb{Q}$ either by Dedekind Cuts
 or by Cauchy Sequences.
 The most usual Completeness Axiom in the analysis textbooks is the existence of supremum for any non-empty and bounded above subset; its first-order version is as follows:

\begin{definition}[$\texttt{\textbf{D-Sup}}$]\label{def:sup}

\noindent
Let $\texttt{\textbf{D-Sup}}$ (Definable Supremum Property) be the following first-order scheme
$$\exists x\,\varphi(x)\,\wedge\,
\exists y\,\forall x\,[\varphi(x)\rightarrow x\!\!\leqslant\!\!y]  \longrightarrow\exists z\,\forall y\,\big(\forall x\,[\varphi(x)\rightarrow x\!\!\leqslant\!\!y] \leftrightarrow z\!\!\leqslant\!\!y  \big)$$
where $\varphi$ is an arbitrary formula.
\hfill\ding{71}\end{definition}

Informally, $\texttt{\textbf{D-Sup}}$ says that for a nonempty ($\exists x\,\varphi(x)$) and bounded from above set (for some $y$ we have $\forall x\,[\varphi(x)\rightarrow x\!\!\leqslant\!\!y]$)  there exists a least upper bound (for some $z$, any $y$ is an upper bound if and only if $y\!\!\geqslant\!\!z$).
A dual statement is the principle of the existence of infimum for any nonempty and bounded from below set:

\begin{definition}[$\texttt{\textbf{D-Inf}}$]\label{def:inf}

\noindent
Let $\texttt{\textbf{D-Inf}}$ be the   first-order scheme
$$\exists x\,\varphi(x)\,\wedge\,
\exists y\,\forall x\,[\varphi(x)\rightarrow y\!\!\leqslant\!\!x]  \longrightarrow\exists z\,\forall y\,\big(\forall x\,[\varphi(x)\rightarrow y\!\!\leqslant\!\!x] \leftrightarrow y\!\!\leqslant\!\!z  \big)$$
for arbitrary formula $\varphi$.
\hfill\ding{71}\end{definition}

The scheme $\texttt{\textbf{D-Inf}}$ is used by Tarski (1940)  for presenting a complete first-order axiomatic system for the ordered field of real numbers  (see   \cite{epstein}). The usual real analytic proof for the equivalence of these two schemes works in  first-order logic as well:

\begin{theorem}[$\texttt{\textbf{D-Sup}}
\Longleftrightarrow\texttt{\textbf{D-Inf}}$]
\label{prop:supinf}

\noindent
In any  linear order  $\langle D;<\rangle$, the scheme \textup{$\texttt{\textbf{D-Sup}}$} holds, if and only if \textup{$\texttt{\textbf{D-Inf}}$} holds.
\end{theorem}
\begin{proof}
We prove only  $\texttt{\textbf{D-Sup}}
\Longrightarrow\texttt{\textbf{D-Inf}}$; its converse can be proved by a dual argument. Suppose that for some   $\varphi(x)$ and $a,b\!\in\!D$ we have (i)~$\varphi(a)$ and (ii)~$\forall x\,[\varphi(x)\rightarrow b\!\!\leqslant\!\!x]$.  For using $\texttt{\textbf{D-Sup}}$,  let $\psi$ be the formula $\psi(x)\!=\!\forall y\!\!<\!\!x\,\neg\varphi(y)$. Then, by (ii), we have (1)~$\psi(b)$ and by (i) we have (2)~$\forall x\,[\psi(x)\rightarrow x\!\!\leqslant\!\!a]$. Now,  by $\texttt{\textbf{D-Sup}}$ there exists some $c\!\in\!D$ such that
\begin{equation}\label{eq:sup}
\forall y\,\big(\forall x\,[\psi(x)\rightarrow x\!\!\leqslant\!\!y] \leftrightarrow c\!\!\leqslant\!\!y  \big).
\end{equation}
We show  $\forall z\,\big(\forall x\,[\varphi(x)\rightarrow z\!\!\leqslant\!\!x] \leftrightarrow z\!\!\leqslant\!\!c\big)$. For every     $z\!\in\!D$ we have:
(I) If $\forall x\,[\varphi(x)\rightarrow z\!\!\leqslant\!\!x]$, then $\psi(z)$; from \eqref{eq:sup},  for $y\!=\!c$,  we have  $\forall x\,[\psi(x)\rightarrow x\!\!\leqslant\!\!c]$, and so   $z\!\!\leqslant\!\!c$.
(II) If $z\!\!\leqslant\!\!c$, then for any $x\!\!<\!\!z$ we have $c\!\not\leqslant\!x$ and so \eqref{eq:sup} implies that for some $u$ we have $\psi(u)$ and $u\!\not\leqslant\!x$. Now, $\psi(u)$ and $x\!\!<\!\!u$ imply that $\neg\varphi(x)$. Thus, $\forall x\!\!<\!\!z\,\neg\varphi(x)$, or equivalently,
 $\forall x\,[\varphi(x)\rightarrow z\!\!\leqslant\!\!x]$.
  \end{proof}

Another completeness principle is the so called Dedekind's Axiom;
which says that there is no proper cut (with a gap) in $\mathbb{R}$. The following is a first-order writing of an equivalent form of this axiom:

\begin{definition}[$\texttt{\textbf{D-Cut}}$]
\label{def:cut}

\noindent
Let $\texttt{\textbf{D-Cut}}$ denote  the first-order scheme
$$\exists x\,\exists y\,[\varphi(x)\wedge\psi(y)] \,\wedge\, \forall x\,\forall  y\,[\varphi(x)\wedge\psi(y)\rightarrow x\!\!<\!\!y]\longrightarrow \exists z\,\forall x\,\forall y\,[\varphi(x)\wedge\psi(y)\rightarrow x\!\!\leqslant\!\!z\!\!\leqslant\!\!y]$$
where  $\varphi$ and $\psi$ are arbitrary formulas.
\hfill\ding{71}\end{definition}

Tarski (1941,1946) used this axiom scheme for presenting another complete first-order axiomatic system for $\mathbb{R}$  (see \cite{epstein}).  Indeed, in  linearly ordered sets,  the axiom  scheme  $\texttt{\textbf{D-Cut}}$ is equivalent with the other completeness axioms schemes of $\texttt{\textbf{D-Sup}}$ and $\texttt{\textbf{D-Inf}}$:

\begin{theorem}[$\texttt{\textbf{D-Sup}}
\Longleftrightarrow\texttt{\textbf{D-Cut}}$]
\label{prop:cutsup}

\noindent
In any   linear order   $\langle D;<\rangle$, the scheme  \textup{$\texttt{\textbf{D-Sup}}$} holds, if and only if  \textup{$\texttt{\textbf{D-Cut}}$} holds.
\end{theorem}
\begin{proof}
If $\texttt{\textbf{D-Sup}}$ holds in a linear order  $\langle D;<\rangle$, then for showing
$\texttt{\textbf{D-Cut}}$ assume that for    $\varphi(x),\psi(x)$ and   $a,b\!\in\!D$ we have (i)~$\varphi(a)\wedge\psi(b)$ and  (ii)~$\forall x\,\forall  y\,[\varphi(x)\wedge\psi(y)\rightarrow x\!\!<\!\!y]$; we show the existence of some $z\!\in\!D$ such that $\forall y\,[\varphi(x)\wedge\psi(y)\rightarrow x\!\!\leqslant\!\!z\!\!\leqslant\!\!y]$. Now, by (i), we have (1)~$\varphi(a)$ and, by (ii), we have (2)~$\forall x\,[\varphi(x)\rightarrow x\!\!<\!\!b]$. By $\texttt{\textbf{D-Sup}}$ there exists some $c\!\in\!D$ such that (3)~$\forall y\,\big(\forall x\,[\varphi(x)\rightarrow x\!\!\leqslant\!\!y] \leftrightarrow c\!\!\leqslant\!\!y  \big)$. We show that 
$\forall x\forall y\,[\varphi(x)\wedge\psi(y)\rightarrow x\!\!\leqslant\!\!c\!\!\leqslant\!\!y]$ holds. By (3), for $y\!=\!c$, we have $x\!\!\leqslant\!\!c$ for any $x$ with $\varphi(x)$. Now assume that $\psi(y)$; then if $c\!\not\leqslant\!y$, by (3), there should exist some $x$ such that $\varphi(x)$ but $y\!\!<\!\!x$. This contradicts (ii) since one the one hand we have $\varphi(x)\wedge\psi(y)$ but on the other hand $y\!\!<\!\!x$!

If $\texttt{\textbf{D-Cut}}$ holds in  a   linear order $\langle D;<\rangle$, then for showing
$\texttt{\textbf{D-Sup}}$ assume that for a formula $\varphi(x)$ and some $a,b\!\in\!D$ we have (i)~$\varphi(a)$ and  (ii)~$\forall x\,[\varphi(x)\rightarrow x\!\!\leqslant\!\!b]$; we show the existence of some $z\!\in\!D$ such that $\forall y\,\big(\forall x\,[\varphi(x)\rightarrow x\!\!\leqslant\!\!y] \leftrightarrow z\!\!\leqslant\!\!y\big)$.
We distinguish two cases:

(I) For some $d\!\in\!D$ we have $\varphi(d)\wedge\forall x [\varphi(x)\rightarrow x\!\!\leqslant\!\!d]$.  In this case, it is easy to show that $\forall y\,\big(\forall x\,[\varphi(x)\rightarrow x\!\!\leqslant\!\!y] \leftrightarrow d\!\!\leqslant\!\!y\big)$ holds, since for any $y$ we have $y\!\!<\!\!d$ if and only if for some $x$ (which is $x\!=\!d$) we have $x\!\!>\!\!y$ and $\varphi(x)$.

(II) We have $\forall y\big[\forall x[\varphi(x)\rightarrow x\!\!\leqslant\!\!y]\rightarrow\neg\varphi(y)\big]$. In this case, put  $\psi(x)\!=\!\forall y\!\!\geqslant\!\!x\,\neg\varphi(y)$. Then, noting that by (ii) we already have $\neg\varphi(b)$, the relation $\psi(b)$ holds, and so by (i) we have  (1)~$\varphi(a)\wedge\psi(b)$. Also, (2)~$\forall x\,\forall  y\,[\varphi(x)\wedge\psi(y)\rightarrow x\!\!<\!\!y]$ holds by the definition of $\psi$.
Thus, by $\texttt{\textbf{D-Cut}}$, there exists some $c\!\in\!D$ such that (3)~$\forall x\,\forall y\,[\varphi(x)\wedge\psi(y)\rightarrow x\!\!\leqslant\!\!c\!\!\leqslant\!\!y]$. We show that this $c$ satisfies $\forall y\,\big(\forall x\,[\varphi(x)\rightarrow x\!\!\leqslant\!\!y] \leftrightarrow c\!\!\leqslant\!\!y  \big)$.

First,  assume that $\forall x\,[\varphi(x)\rightarrow x\!\!\leqslant\!\!y]$; then $\neg\varphi(y)$ holds by the assumption (II).
So, $\psi(y)$ holds and therefore $c\!\!\leqslant\!\!y $ by (3).

Second, assume that $c\!\!\leqslant\!\!y$ but $\varphi(x)\wedge x\!\not\leqslant\!y $ for some $x$. Then by (3), and $\varphi(x)$,  we should have $x\!\!\leqslant\!\!c$ but this contradicts $c\!\!\leqslant\!\!y\!\!<\!\!x$!
 \end{proof}

Finally, we observe that continuous induction is equivalent to (any of) the (above) completeness axiom(s), even in first-order logic:

\begin{theorem}[$\texttt{\textbf{D-Inf}}
\Longleftrightarrow\texttt{\textbf{DCI}}$]
\label{prop:infind}

\noindent
In any  linear order  $\langle D;<\rangle$, the scheme \textup{$\texttt{\textbf{D-Inf}}$} holds, if and only if \textup{$\texttt{\textbf{DCI}}$} holds.
\end{theorem}
\begin{proof}
Suppose that $\texttt{\textbf{D-Inf}}$ holds in the   linear order   $\langle D;<\rangle$. For proving  \texttt{\textbf{DCI}},  assume that for  $\varphi$ and some $a\!\in\!D$ we have (i)~$\forall y\!\!<\!\!a\,\varphi(y)$ and (ii)~$\forall x\big[\forall y\!\!<\!\!x\,\varphi(y)\rightarrow\exists z\!\!>\!\!x\,\forall y\!\!<\!\!z\,\varphi(y)\big]$; then we show that $\forall x\,\varphi(x)$. If for some $b\!\in\!D$ we had $\neg\varphi(b)$, then put $\varphi'(x)\!=\!\neg\varphi(x)$. Now we have (1)~$\varphi'(b)$ and (2)~$\forall x\,[\varphi'(x)\rightarrow a\!\!\leqslant\!\!x]$ by (i).
By the assumption $\texttt{\textbf{D-Inf}}$ there exists some element $c\!\in\!D$ for which we have

$\forall y\,\big(\forall x\,[\varphi'(x)\rightarrow y\!\!\leqslant\!\!x] \leftrightarrow y\!\!\leqslant\!\!c  \big)$;

\noindent
or, equivalently, (3)~$\forall y\,[y\!\!\leqslant\!\!c  \leftrightarrow \forall x\!\!<\!\!y\,\varphi(x)]$. Thus, $\forall x\!\!<\!\!c\,\varphi(x)$ and so by (ii) there exists some $d\!\in\!D$ such that $d\!\!>\!\!c$ and $\forall y\!\!<\!\!d\,\varphi(y)$. By (3), since  $d\!\!>\!\!c$,  there exists some $x\!\!<\!\!d$ such that $\neg\varphi(x)$. This is a contradiction, since we had $\forall y\!\!<\!\!d\,\varphi(y)$.

Now, suppose that $\texttt{\textbf{DCI}}$ holds in a linear order $\langle D;<\rangle$; for proving $\texttt{\textbf{D-Inf}}$ assume that we have  (i)~$\varphi(a)$ and (ii)~$\forall x\,[\varphi(x)\rightarrow b\!\!\leqslant\!\!x]$ for some formula $\varphi$ and some $a,b\!\in\!D$. We aim at showing the existence of some $z\!\in\!D$ such that (iii)~$\forall y\,\big(\forall x\,[\varphi(x)\rightarrow y\!\!\leqslant\!\!x] \leftrightarrow y\!\!\leqslant\!\!z  \big)$. Assume that for no $z$ the condition (iii) holds; so for any $z$ there exists some $z'$ such that (iv)~$\neg[z'\!\!\leqslant\!\!z\leftrightarrow\forall x\!\!<\!\!z'\neg\varphi(x)]$. Now, for the formula $\varphi'(x)\!=\!\neg\varphi(x)$ we have (1)~$\forall x\!\!<\!\!b\,\varphi'(x)$ by (ii), and we show that the condition

(2)~$\forall x\big[\forall y\!\!<\!\!x\,\varphi'(y)\rightarrow\exists z\!\!>\!\!x\,\forall y\!\!<\!\!z\,\varphi'(y)\big]$

\noindent
holds: for any fixed $x_0$ if $\forall y\!\!<\!\!x_0\,\varphi'(y)$, then (3)~$\forall x\!\!<\!\!x_0\,\neg\varphi(x)$. There exists some $z_0$ such that (iv) holds (for $z'=z_0$); if $z_0\!\!\leqslant\!\!x_0$,  then $\neg\forall x\!\!<\!\!z_0\neg\varphi(x)$ or $\exists x\!\!<\!\!z_0\varphi(x)$ which contradicts (3).
So, we have $x_0\!\!<\!\!z_0$; by (iv) we should have  $\forall x\!\!<\!\!z_0\varphi'(x)$. Whence, (2) holds, and we can apply $\texttt{\textbf{DCI}}$; which implies that $\forall x\,\varphi'(x)$ and this contradicts   (i). Therefore, for some $z\!\in\!D$, (iii) should hold.
  \end{proof}

It is proved in  \cite[Theorems~3.1,3.2]{kalantari}
that  continuous induction  $\texttt{\textbf{CI}}$ is equivalent to the   infimum (greatest lower bound) property  in ordered fields; also, $\texttt{\textbf{CI}}$ is proved in \cite[Theorem~1]{jing} to be equivalent with Dedekind's Axiom (the nonexistence of cuts with gaps).
We now show that the principle of  continuous induction in Definition~\ref{def:ind''} ($\texttt{\textbf{CI}}_1$) is actually equivalent to the Archimedean property in ordered abelian groups.

\begin{definition}[Archimedean Property, $\texttt{\textbf{AP}}$]\label{def:arch}

\noindent
An ordered abelian group $\langle G;+,0,<\rangle$ has the Archimedean property when for any $a,b\!\in\!G$ with $a\!\!>\!\!0$ there exists some $n\!\in\!\mathbb{N}$ such that $b\!\!<\!\!n\centerdot a$, where $n\centerdot a\!=\!\underbrace{a+\cdots+a}_{n-\text{times}}$.
\hfill\ding{71}\end{definition}

\begin{theorem}[$\texttt{\textbf{AP}}
\iff\texttt{\textbf{CI}}_1$]\label{th:arch}

\noindent An ordered abelian group
  has the Archimedean property
  if and only if it
satisfies   \textup{$\texttt{\textbf{CI}}_1$}.
\end{theorem}
\begin{proof}
Suppose that the ordered abelian group $\langle G;+,0,<\rangle$ has the Archimedean property ($\texttt{\textbf{AP}}$) and
we have
(i)~$]\!-\infty,a]\subseteq\mathcal{A}$ and
(ii)~for any $x\!\in\!G$, if $x\!\in\!\mathcal{A}$, then $[x,x\!+\!\epsilon]\subseteq\mathcal{A}$;
for a subset $\mathcal{A}\subseteq G$ and  some $\epsilon,a\!\in\!G$ with $\epsilon\!\!>\!\!0$.
By induction on $n\!\in\!\mathbb{N}$ we can show that (iii)~$]\!-\infty,a\!+\!n\centerdot\epsilon]\subseteq\mathcal{A}$:
for $n\!=\!0$ it follows from (i) and the induction step
follows from (ii). For an arbitrary $x\!\in\!G$, by $\texttt{\textbf{AP}}$ there exists some $n\!\in\!\mathbb{N}$ such that $x\!-\!a\!\!<\!\!n\centerdot\epsilon$;  so we have $x\!\in\!\mathcal{A}$ by (iii). Whence, $\mathcal{A}\!=\!G$; and so $\texttt{\textbf{CI}}_1$ holds in $G$.

Now, suppose that the ordered abelian group $\langle G;+,0,<\rangle$ satisfies $\texttt{\textbf{CI}}_1$. For any $a,b\!\in\!G$ with $a\!\!>\!\!0$ let $\mathcal{B}\!=\!\{x\!\in\!G\mid \exists n\!\in\!\mathbb{N}\!:x\!\!<\!\!n\centerdot a\}$. Then we obviously have (i)~$]\!-\infty,0]\subseteq\mathcal{B}$; we show that (ii)~for any $x\!\in\!G$, if $x\!\in\!\mathcal{B}$, then $[x,x\!+\!a]\subseteq\mathcal{B}$. For any $x\!\in\!\mathcal{B}$ we have $x\!\!<\!\!m\centerdot a$ for some $m\!\in\!\mathbb{N}$; so for any $y\!\in\![x,x\!+\!a]$ we have $y\!\!\leqslant\!\!x\!+\!a\!\!<\!\!(m\!+\!1)\centerdot a$, and so $y\!\in\!\mathcal{B}$. Thus, by $\texttt{\textbf{CI}}_1$, from (i) and (ii), we should have $\mathcal{B}\!=\!G$, and so for any $b\!\in\!S$ there exists some $n\!\in\!\mathbb{N}$ such that $b\!\!<\!\!n\centerdot a$; whence $G$ has $\texttt{\textbf{AP}}$.
  \end{proof}

Let us note that  Theorem~\ref{th:arch} gives another proof for  $\mathbb{Q}\vDash\texttt{\textbf{CI}}_1$ (and also $\mathbb{Q}\vDash\texttt{\textbf{DCI}}_1$) that was proved already in Theorem~\ref{th:indind''}.
So, $\texttt{\textbf{CI}}_1$ can be added to the list of 42 equivalent forms of the Archimedean property (of ordered fields)  in \cite{deteisman}. Whence, while the principle of  continuous induction in Definitions~\ref{def:ind} ($\texttt{\textbf{CI}}$) and \ref{def:ind'} ($\texttt{\textbf{CI}}_2$) are equivalent with  the completeness axiom (in the ordered fields), the principle $\texttt{\textbf{CI}}_1$ is equivalent with the Archimedean property (and is not equivalent with the completeness axiom) in such fields.

\section{Real Closed Fields, a First-Order Logical Study}\label{sec:rcf}

What makes the Real Closed Fields interesting in Mathematical Logic is Tarski's Theorem that,  ${\rm Th}(\langle\mathbb{R};+,0,-,<,\times,1\rangle)$, the complete first-order theory of the ordered field of real numbers is decidable and axiomatizable by the theory of real closed (ordered) fields (see e.g. \cite{kk}) . The real closed fields are studied thoroughly in Algebra  and Algebraic Geometry (see e.g. \cite{basu}); here we study some equivalent  axiomatizations of the first-order theory of real closed fields. The most usual definition of a real closed field is an ordered field which satisfies the following axiom scheme:

\begin{definition}[$\texttt{\textbf{RCF}}$]\label{def:rcf}

\noindent
Let $\texttt{\textbf{RCF}}$ be the following axiom scheme:
$$\forall x\!\!>\!\!0\,\exists y\,(y^2\!=\!x) \;\wedge\; \forall \{a_i\}_{i\leqslant 2n}\,\exists x\,  (x^{2n+1}\!+\!\sum_{i\leqslant 2n}a_ix^{i}\!=\!0),$$
for $n\!\in\!\mathbb{N}$ with $n\!\!>\!\!1$, together with the axioms of Ordered Fields ($\texttt{\textbf{OF}}$).
\hfill\ding{71}\end{definition}

By definition,  a real closed field is an ordered field in which every positive element has a square root and  every odd-degree polynomial has a root in it. This most usual definition of real closed fields is in fact the most inapplicable one, since many interesting theorems on and about real closed fields use other equivalent definitions. The most applicable (and the most fruitful) definition of real closed fields is the one which says that a real closed field is an ordered field in which the Intermediate Value Theorem (IVT) holds:

\begin{definition}[$\texttt{\textbf{IVT}}$]\label{def:ivt}

\noindent
Let $\texttt{\textbf{IVT}}$ be the  first-order scheme
$$\forall\mathfrak{p}\,\forall u,v\,\exists x\,  [u\!\!<\!\!v \wedge \mathfrak{p}(u)\!\cdot\!\mathfrak{p}(v)\!\!<\!\!0
\longrightarrow u\!\!<\!\!x\!\!<\!\!v \wedge \mathfrak{p}(x)\!=\!0],$$
where $\forall\mathfrak{p}$ stands for $\forall\{a\}_{i\leqslant m}$ for polynomial  $\mathfrak{p}(x)\!=\!\sum_{i\leqslant m} a_ix^{i}$.
\hfill\ding{71}\end{definition}

Let us note that the terms of the language of ordered fields $\{+,0,-,<,\times,1\}$ which contain a variable $x$ are equal to polynomials like $\mathfrak{p}(x)\!=\!\sum_{i\leqslant m}a_ix^i$ where $a_i$'s are some $x$-free terms.
For the usefulness of $\texttt{\textbf{IVT}}$ we first note that it implies   $\texttt{\textbf{RCF}}$:

\begin{theorem}[$\texttt{\textbf{IVT}}
\Longrightarrow\texttt{\textbf{RCF}}$]\label{prop:ivtrcf}

\noindent
In any ordered field that \textup{$\texttt{\textbf{IVT}}$} holds,
\textup{$\texttt{\textbf{RCF}}$} holds too.
\end{theorem}
\begin{proof}
For any $a\!\!>\!\!0$ let $\mathfrak{p}(x)\!=\!x^2\!-\!a$. Then we have $\mathfrak{p}(0)\!=\!-a\!\!<\!\!0\!\!<\!\!
a^2\!+\!\frac{1}{4}\!=\!
\mathfrak{p}(a\!+\!\frac{1}{2}),$  and so by $\texttt{\textbf{IVT}}$ for some $x$ with $0\!\!<\!\!x\!\!<\!\!a\!+\!\frac{1}{2}$ we have $\mathfrak{p}(x)\!=\!0$; thus, $x^2\!=\!a$.

We can write  any odd-degree polynomial  as $\mathfrak{p}(x)\!=\!x^{2n+1}\!+\!\sum_{i\leqslant 2n}a_ix^{i}$ by multiplying with an $x$-free term, if necessary. Let $v\!=\!1\!+\!\sum_{i\leqslant 2n}|a_i|$ and $u\!=\!-v$. Then $u\!\leqslant\!-1\!<\!0\!<\!1\!\leqslant\!v$, so $|u|,|v|\!\geqslant\!1$. Now,  by $u\!+\!\sum_{i}|a_i|\!=\!\!-\!1$ and the triangle inequality, we have

$\mathfrak{p}(u)\!\leqslant\!  u^{2n+1}\!+\!\sum_{i\leqslant 2n}|a_i||u|^{i}\!\leqslant\!u^{2n+1}\!+\!\sum_{i\leqslant 2n}|a_i|u^{2n}\!=\!u^{2n}(u\!+\!\sum_{i\leqslant 2n}|a_i|)
 \!=\!-u^{2n}\!\!<\!\!0,$ and

 $\mathfrak{p}(v)\!\geqslant\! v^{2n+1}\!-\!\sum_{i\leqslant 2n}|a_i||v|^{i}\!\geqslant\!v^{2n+1}\!-\!\sum_{i\leqslant 2n}|a_i|v^{2n}\!=\!v^{2n}(v\!-\!\sum_{i\leqslant 2n}|a_i|)\!=\!v^{2n}\!\!>\!\!0.$

\noindent
 Now, the desired conclusion follows from $\texttt{\textbf{IVT}}$ by $u\!\!<\!\!v$ and $\mathfrak{p}(u)\!\!<\!\!0\!\!<\!\!\mathfrak{p}(v)$.
 \end{proof}

As some other uses of the $\texttt{\textbf{IVT}}$ let us take a look at some real analytic theorems that are proved algebraically (cf. \cite{basu}).

\begin{definition}[Derivative]\label{def:der}

\noindent
The derivative of a polynomial $\mathfrak{p}(x)\!=\!\sum_{i=0}^na_ix^i$ is the polynomial $\mathfrak{p}'(x)\!=\!\sum_{i=1}^nia_ix^{i-1}$.
\hfill\ding{71}\end{definition}

\begin{remark}[Some Properties of  Derivatives]\label{rem:der}

\noindent
For polynomials $\mathfrak{p}(x)\!=\!\sum_{i\leqslant n}a_ix^i$ and $\mathfrak{q}(x)\!=\!\sum_{j\leqslant m}b_jx^j$ we have:
\begin{itemize}
\item $(a\mathfrak{p})'\!=\!a\mathfrak{p}'$ for a constant (i.e., an $x$-free term) $a$,
\item $(\mathfrak{p}\!+\!\mathfrak{q})'\!=\!\mathfrak{p}'\!+\!\mathfrak{q}'$, and
\item $(\mathfrak{p}\cdot\mathfrak{q})'
    \!=\!(\mathfrak{p}'\cdot\mathfrak{q})
    \!+\!(\mathfrak{p}\cdot\mathfrak{q}')$;
\end{itemize}
which can be verified rather easily. For example, the last item can be verified by noting that the coefficient of $x^{k\!-\!1}$ in $(\mathfrak{p}\cdot\mathfrak{q})'$ is $k\sum_{i+j=k}a_ib_j\!=\!\sum_{i=0}^kka_ib_{k\!-\!i}
\!=\!\sum_{i=0}^k[ia_ib_{k\!-\!i}\!+\!(k\!-\!i)a_ib_{k\!-\!i}]
\!=\![\sum_{i=0}^k(ia_i)b_{k\!-\!i}]
\!+\![\sum_{j=0}^ka_{k\!-\!j}(jb_{j})]$ in which the first summand is the coefficient of $x^{k\!-\!1}$ in $(\mathfrak{p}'\cdot\mathfrak{q})$ and the second summand is   the coefficient of $x^{k\!-\!1}$ in $(\mathfrak{p}\cdot\mathfrak{q}')$.
\hfill\ding{71}\end{remark}

It goes without saying that the properties mentioned in  Remark~\ref{rem:der} are the ones that are usually  learnt in elementary calculus  through the analytic methods (cf. \cite{spivak}); so are the following theorem and lemma.

\begin{theorem}[$\texttt{\textbf{IVT}}
\Longrightarrow\texttt{\textbf{Rolle}}
 + \texttt{\textbf{MVT}}
 + \texttt{\textbf{Derivative Signs}}$]\label{prop:ivt}

\noindent
Let $F$ be an ordered field in which \textup{$\texttt{\textbf{IVT}}$} holds. Let $\mathfrak{p}(x)$ be a polynomial with the coefficients in $F$ and let $a,b\!\in\!F$ with $a\!\!<\!\!b$.
\begin{itemize}\itemindent=8.75em
\item[{\bf Rolle's Theorem:}] If  $\mathfrak{p}(a)\!=\!\mathfrak{p}(b)\!=\!0$, then  for some $c\!\in]a,b[$ we have  $\mathfrak{p}'(c)\!=\!0$.
\item[{\bf Mean Value Theorem:}] There exists some $c\!\!\in]a,b[$ such that $\mathfrak{p}'(c)\!=\!    \dfrac{\mathfrak{p}(b)\!-\!\mathfrak{p}(a)}{b-a}$.
\item[{\bf Derivative Signs:}] If $\mathfrak{p}'(x)\!\!>\!\!0$ \textup{(respectively, $\mathfrak{p}'(x)\!\!<\!\!0$)} for all $x\!\!\in]a,b[$, then $\mathfrak{p}(u)\!\!<\!\!\mathfrak{p}(v)$ \textup{(respectively, $\mathfrak{p}(u)\!\!>\!\!\mathfrak{p}(v)$)} for any $u,v\!\in\!F$ with  $a\!\!<\!\!u\!\!<\!\!v\!\!<\!\!b$.
\end{itemize}
\end{theorem}
\begin{proof}
 For
{\sf Rolle's Theorem} we note that since any polynomial can have only a finite number of roots, then we can assume that there is no root of $\mathfrak{p}$ in the open interval $]a,b[$. Now, by the assumption $\mathfrak{p}(a)\!=\!\mathfrak{p}(b)\!=\!0$, both $(x\!-\!a)$ and $(x\!-\!b)$  divide $\mathfrak{p}(x)$; let $m$ be the greatest natural number such that $(x\!-\!a)^m$  divides $\mathfrak{p}(x)$ and $n$ be the greatest number such that $(x\!-\!b)^n$  divides $\mathfrak{p}(x)$. Then we can write $\mathfrak{p}(x)\!=\!(x\!-\!a)^m(x\!-\!b)^n\mathfrak{q}(x)$
for some polynomial $\mathfrak{q}(x)$ that has no root in $]a,b[$, and so by $\texttt{\textbf{IVT}}$ we have $\mathfrak{q}(a)\mathfrak{q}(b)\!\!>\!\!0$. Therefore, we have, by Remark~\ref{rem:der}, that  $\mathfrak{p}'(x)\!=\!(x\!-\!a)^{m-1}(x\!-\!b)^{n-1}
\mathfrak{r}(x)$ where
$\mathfrak{r}(x)\!=\!m(x\!-\!b)\mathfrak{q}(x)\!+\!
n(x\!-\!a)\mathfrak{q}(x)\!+\!(x\!-\!a)(x\!-\!b)
\mathfrak{q}'(x)$. Whence, 
$\mathfrak{r}(a)\mathfrak{r}(b)
\!=\!m(a\!-\!b)\mathfrak{q}(a)n(b\!-\!a)\mathfrak{q}(b)
\!=\!\!-\!mn(b\!-\!a)^2\mathfrak{q}(a)\mathfrak{q}(b)\!\!<\!\!0$,
  and so $\texttt{\textbf{IVT}}$ implies the existence of some $c\!\in]a,b[$ with  $\mathfrak{r}(c)\!=\!0$; from which  $\mathfrak{p}'(c)\!=\!0$ follows.

The  {\sf Mean Value Theorem}  follows, classically and standardly,  from {\sf Rolle's Theorem} for the polynomials
$\mathfrak{q}(x)\!=\!\mathfrak{p}(x)\!-\!
\frac{\mathfrak{p}(b)\!-\!\mathfrak{p}(a)}{b\!-\!a}x$, and $\mathfrak{r}(x)\!=\!\mathfrak{q}(x)\!-\!\mathfrak{q}(a)$,  since $\mathfrak{q}(a)\!=\!\mathfrak{q}(b)$ and so
$\mathfrak{r}(a)\!=\!\mathfrak{r}(b)\!=\!0$; also, $\mathfrak{r}'(x)\!=\!\mathfrak{q}'(x)
\!=\!\mathfrak{p}'(x)\!-\!\frac{\mathfrak{p}(b)-\mathfrak{p}(a)}{b-a}$.

{\sf Derivative Signs} easily follows from the  {\sf Mean Value Theorem}: for any such $u,v$ there exists some $w\!\in]u,v[$ such that $\mathfrak{p}(v)\!-\!\mathfrak{p}(u)
\!=\!\mathfrak{p}'(w)\cdot(v\!-\!u)
$. Now, if $\mathfrak{p}'(w)\!\!>\!\!0$, then $\mathfrak{p}(v)\!\!>\!\!\mathfrak{p}(u)$, and if $\mathfrak{p}'(w)\!\!<\!\!0$, then $\mathfrak{p}(v)\!\!<\!\!\mathfrak{p}(u)$.
 \end{proof}

The last (and the strongest) witness for  the strength of $\texttt{\textbf{IVT}}$ is Tarski's Theorem that proves that the theory $\texttt{\textbf{OF}}\!+\!\texttt{\textbf{IVT}}$ is complete; i.e., for any sentence $\theta$ in the language of ordered fields either we have $\texttt{\textbf{OF}}\!+\!\texttt{\textbf{IVT}}\vdash\theta$ or we have
$\texttt{\textbf{OF}}\!+\!\texttt{\textbf{IVT}}\vdash\neg\theta$.
In the Appendix we present a somewhat modified proof of Kreisel and Krivine \cite{kk} for this result (Theorem~\ref{prop:ivtcomp}). Let us then note this last result implies that every statement that is true in $\mathbb{R}$ is provable from $\texttt{\textbf{IVT}}$ ($\!+\texttt{\textbf{OF}}$). Since, otherwise its negation would have been provable, and this would contradict its truth in $\mathbb{R}$. Thus, for example we have $\texttt{\textbf{IVT}}
\Longrightarrow\texttt{\textbf{D-Sup}}$ (and also $\texttt{\textbf{IVT}}
\Longrightarrow\texttt{\textbf{DCI}}$ and
$\texttt{\textbf{IVT}}
\Longrightarrow\texttt{\textbf{D-Cut}}$, etc.) over the theory $\texttt{\textbf{OF}}$.   Moreover, if an axiom (or an axiom scheme)  that is true in $\mathbb{R}$ can prove $\texttt{\textbf{IVT}}$, then it is actually equivalent with $\texttt{\textbf{IVT}}$; and so can be used as another axiomatization of the theory of real closed fields (over  $\texttt{\textbf{OF}}$). So, by the following theorem (\ref{prop:infivt}), all the schemes that we have considered so far (except $\texttt{\textbf{DCI}}_1$) are equivalent with $\texttt{\textbf{IVT}}$. We will need the following lemma for proving $\texttt{\textbf{D-Inf}}
\Longrightarrow\texttt{\textbf{IVT}}$.

\begin{lemma}[Continuity of Polynomials]\label{lem:cont}

\noindent
For any polynomial $\mathfrak{q}$ and element  $w$, if $\mathfrak{q}(w)\!\!>\!\!0$ \textup{(respectively,  if $\mathfrak{q}(w)\!\!<\!\!0$)}, then there exists some $\epsilon\!\!>\!\!0$ such that for any $x\!\in\![w\!-\!\epsilon,w\!+\!\epsilon]$ we have $\mathfrak{q}(x)\!\!>\!\!0$ \textup{(respectively, $\mathfrak{q}(x)\!\!<\!\!0$)}.
\end{lemma}
\begin{proof}
Let $\mathfrak{q}(x)\!=\!\sum_{k\leqslant m}a_kx^k$; and suppose that for some $A\!\!>\!\!0$ we have $|w|\!\!<\!\!A$ and $|a_k|\!\!<\!\!A$ for $k\!\!\leqslant\!\!m$. Let $B\!=\!A\sum_{k\leqslant  m}\sum_{i<k}{k\choose i}A^i$, and choose an $\epsilon\!\!>\!\!0$ such that $\epsilon\!\!<\!\!1$ and  $\epsilon B\!\!<\!\!\frac{1}{2}|\mathfrak{q}(w)|$; note that $|\mathfrak{q}(w)|\!\neq\!0$. Now, for any $x\!\in\![w\!-\!\epsilon,w\!+\!\epsilon]$ we have $x\!=\!w\!+\!\delta$ for some $\delta\!\in\![\!-\!\epsilon,\epsilon]$. So,

\begin{tabular}{lllll}
$|\mathfrak{q}(x)-\mathfrak{q}(w)|$ & $=$ & $|\sum_{k\leqslant m}a_k([w\!+\!\delta]^k-w^k)|$
& $=$ & $|\sum_{k\leqslant m}a_k\sum_{i<k}{k\choose i}\delta^{k-i}w^i|$ \\
& $\leqslant$ & $\sum_{k\leqslant m}|a_k|\sum_{i<k}{k\choose i}|\delta|^{k-i}|w|^i$
& $\leqslant$ &  $\sum_{k\leqslant m}A\sum_{i<k}{k\choose i}\epsilon^{k-i}A^i$ \\
& $=$ & $\epsilon A\sum_{k\leqslant m}\sum_{i<k}{k\choose i}\epsilon^{k-1-i}A^i$
& $\leqslant$ & $\epsilon B$ \;\qquad
 $<$ \;\qquad   $\frac{1}{2}|\mathfrak{q}(w)|.$
\end{tabular}

\noindent
Whence, $\mathfrak{q}(w)\!-\!\frac{1}{2}|\mathfrak{q}(w)|\!\!<\!\!
\mathfrak{q}(x)\!\!<\!\!\mathfrak{q}(w)\!+\!
\frac{1}{2}|\mathfrak{q}(w)|$. So, if $\mathfrak{q}(w)\!\!>\!\!0$, then $\mathfrak{q}(x)\!\!>\!\!\frac{1}{2}
\mathfrak{q}(w)\!\!>\!\!0$ and if $\mathfrak{q}(w)\!\!<\!\!0$, then $\mathfrak{q}(x)\!\!<\!\!\frac{1}{2}
\mathfrak{q}(w)\!\!<\!\!0$, for any $x\!\in\![w\!-\!\epsilon,w\!+\!\epsilon]$.
 \end{proof}

\begin{theorem}[$\texttt{\textbf{D-Inf}}
\Longrightarrow\texttt{\textbf{IVT}}$]
\label{prop:infivt}

\noindent
Any ordered field that satisfies \textup{$\texttt{\textbf{D-Inf}}$} satisfies \textup{$\texttt{\textbf{IVT}}$} too.
\end{theorem}
\begin{proof}
Suppose that for  $\mathfrak{p}$ and $u,v$ with $u\!\!<\!\!v$ we have
$\mathfrak{p}(u)\mathfrak{p}(v)\!\!<\!\!0$.
Let $\varphi(x)\!=\!
u\!\!\leqslant\!\!x\!\!\leqslant\!\!v\wedge \mathfrak{p}(u)\mathfrak{p}(x)\!\!<\!\!0$. Then $\varphi(v)$ and $\forall x\,[\varphi(x)\rightarrow u\!\!\leqslant\!\!x]$. So, by $\texttt{\textbf{D-Inf}}$ there exists some $w$ such that (1)~$\forall y\,\big(\forall x\,[\varphi(x)\rightarrow y\!\!\leqslant\!\!x] \leftrightarrow y\!\!\leqslant\!\!w  \big)$. We show that $u\!\!\leqslant\!\!w\!\!\leqslant\!\!v$ and $\mathfrak{p}(w)\!=\!0$.
If $w\!\!<\!\!u$, then by (1) for some $x$ we should have $\varphi(x)$ and $x\!\!<\!\!u$, and this is impossible; so $u\!\!\leqslant\!\!w$. Also, if $v\!\!<\!\!w$, then there exists some $w'$ such that $v\!\!<\!\!w'\!\!<\!\!w$; and so by (1) for any $x$ with $\varphi(x)$ we should have $w'\!\!\leqslant\!\!x$ and in particular, since $\varphi(v)$, we should have $w'\!\!\leqslant\!\!v$; contradiction, thus $w\!\!\leqslant\!\!v$. Now, we show that $\mathfrak{p}(u)\mathfrak{p}(w)\!=\!0$.  If not, i.e., $\mathfrak{p}(u)\mathfrak{p}(w)\!\neq\!0$,
then we have  either (i)~$\mathfrak{p}(u)\mathfrak{p}(w)\!\!<\!\!0$ or (ii)~$\mathfrak{p}(u)\mathfrak{p}(w)\!\!>\!\!0$. In either case by Lemma~\ref{lem:cont} there exists some $\epsilon\!\!>\!\!0$ such that either (i')~$\forall x\!\in\![w\!-\!\epsilon,w\!+\!\epsilon]\!\!:  \mathfrak{p}(u)\mathfrak{p}(x)\!\!<\!\!0$ or
(ii')~$\forall x\!\in\![w\!-\!\epsilon,w\!+\!\epsilon]\!\!:  \mathfrak{p}(u)\mathfrak{p}(x)\!\!>\!\!0$.
 In case (i') we have $\varphi(w\!-\!\epsilon)$ but then by (1), for $y\!=\!w$, we should have $w\!\!\leqslant\!\!w\!-\!\epsilon$; a contradiction.
In case (ii')  by (1), since $w\!+\!\epsilon\!\not\leqslant\!w$, there should exist some $x$ such that $\varphi(x)$ and $x\!\!<\!\!w\!+\!\epsilon$. By (1), for $y\!=\!w$, we should have $w\!\!\leqslant\!\!x$, from $\varphi(x)$. So, $x\!\in\![w\!-\!\epsilon,w\!+\!\epsilon]$; but then we should have by (ii') that $\mathfrak{p}(u)\mathfrak{p}(x)\!\!>\!\!0$   contradicting  $\varphi(x)$ (which implies $\mathfrak{p}(u)\mathfrak{p}(x)\!\!<\!\!0$).  Thus, $\mathfrak{p}(u)\mathfrak{p}(w)\!=\!0$;  and so $\mathfrak{p}(w)\!=\!0$. Let us note that $w\!\neq\!u,v$ by $\mathfrak{p}(u)\mathfrak{p}(v)\!\!<\!\!0$; and so $w\!\in]u,v[$.
 \end{proof}

\begin{remark}[\textup{\bf $\texttt{\textbf{D-Inf}}
\Longrightarrow\texttt{\textbf{RCF}}$}]\label{rem:infrcf}

\noindent
Theorems~\ref{prop:infivt} and  Theorem~\ref{prop:ivtrcf} imply that any ordered field that satisfies $\texttt{\textbf{D-Inf}}$ satisfies $\texttt{\textbf{RCF}}$ too. This could also be proved directly by the elementary analytic way (using Lemma~\ref{lem:cont}): for a given $a\!\!>\!\!0$, the infimum $x$ of the (definable) set $\{u\mid u\!\!>\!\!0\wedge u^2\!\!>\!\!a\}$ satisfies $x^2=a$, and the infimum $y$ of the set $\{u\mid (u^{2n+1}\!+\!\sum_{i\leqslant 2n}a_iu^{i})\!>\!0\}$ satisfies $y^{2n+1}\!+\!\sum_{i\leqslant 2n}a_iy^{i}\!=\!0$.
\hfill\ding{71}\end{remark}

So, from Theorem~\ref{prop:ivtcomp} we will have   $\texttt{\textbf{IVT}}
\equiv\texttt{\textbf{D-Inf}}\equiv
\texttt{\textbf{D-Sup}}
\equiv\texttt{\textbf{D-Cut}}
\equiv\texttt{\textbf{DCI}}
\equiv\texttt{\textbf{DCI}}_2$.
As for $\texttt{\textbf{RCF}}$, we have shown only the  conditional $\texttt{\textbf{IVT}}
\Longrightarrow\texttt{\textbf{RCF}}$ in Theorem~\ref{prop:ivtrcf}. Its converse ($\texttt{\textbf{RCF}}
\Longrightarrow\texttt{\textbf{IVT}}$) is usually proved in the literature by first proving the Fundamental Theorem of Algebra; which says that the field
$\mathbb{R}(\imath^{^{\!\!\!\centerdot}})$
 is algebraically closed, where  $F(\imath^{^{\!\!\!\centerdot}})$ is the result of adjoining $\imath^{^{\!\!\!\centerdot}}\!=\!\sqrt{\!-\!1}$ to $F$ (see e.g. \cite[Theorem~2.11]{basu}).  A field is called algebraically closed when  any polynomial with coefficients in it has a root (in that field).
As a matter of fact, an equivalent definition for real closed fields in the literature is that
``an ordered field $F$ is real closed if and only if the field $F(\imath^{^{\!\!\!\centerdot}})$ is algebraically closed''.
An equivalent statement is that ``{\sl an ordered field is real closed if and only if every positive element has a square root and any polynomial can be factorized into linear and quadratic factors}''. This is also an equivalent form of the   Fundamental Theorem of Algebra (for $\mathbb{R}$); for which we propose the following first-order scheme:

\begin{definition}[$\texttt{\textbf{FTA}}_{\textsf{RCF}}$]
\label{def:fta}

\noindent
Let $\texttt{\textbf{FTA}}_{\textsf{RCF}}$ denote  the conjunction of $\forall x\!\!>\!\!0\,\exists y\,(y^2\!=\!x) $ and the following first-order scheme
$$\forall \{a_i\}_{i<2n}\,\exists \{b\}_{j<n}\,\exists \{c\}_{j<n}\,\forall x\,[(x^{2n}\!+\!\sum_{i<2n}a_ix^{i})\!=\!
\prod_{j<n}(x^2\!+\!b_jx\!+\!c_j)]$$
for any $n\!\in\!\mathbb{N}$ with $n\!\!>\!\!1$.
\hfill\ding{71}\end{definition}
So, $\texttt{\textbf{FTA}}_{\textsf{RCF}}$ says that any even-degree polynomial can be factorized as a product of  some quadratic polynomials.
The following theorem is proved in  e.g. \cite[Theorem~2.11]{basu}; here we present a different proof:

\begin{theorem}[$\texttt{\textbf{FTA}}_{\textsf{RCF}}
\Longrightarrow\texttt{\textbf{IVT}}$]\label{th:ftaivt}

\noindent
Any ordered field that satisfies \textup{$\texttt{\textbf{FTA}}_{\textsf{RCF}}$} satisfies \textup{$\texttt{\textbf{IVT}}$} too.
\end{theorem}
\begin{proof}  For a polynomial
$\mathfrak{p}(x)$ with degree $m$, suppose that $\mathfrak{p}(u)\cdot\mathfrak{p}(v)\!\!<\!\!0$ for some $u,v$ with $u\!\!<\!\!v$. Put
$\mathfrak{q}(x)\!=\!\frac{1}{\mathfrak{p}(u)}(1\!+\!x^2)^m\,
\mathfrak{p}(u\!+\!\frac{v-u}{1+x^2})$. Then $\mathfrak{q}(x)\!=\!x^{2m}\!+\!\mathfrak{r}(x^2)$ for a polynomial $\mathfrak{r}(x)$ whose degree is less than $m$.
By  $\texttt{\textbf{FTA}}_{\textsf{RCF}}$ we have
$\mathfrak{q}(x)\!=\!\prod_{j<m}(x^2\!+\!b_jx\!+\!c_j)$ for
 some $\{b_j\}_{j<m}$ and $\{c_j\}_{j<m}$.
Now, $\prod_{j<m}c_j\!=\!\mathfrak{q}(0)
\!=\!\frac{\mathfrak{p}(v)}{\mathfrak{p}(u)}
\!=\!\frac{\mathfrak{p}(u)\mathfrak{p}(v)}{\mathfrak{p}(u)^2}
\!\!<\!\!0$. So, $c_j\!\!<\!\!0$ for some $j$.
Assume, $c_0\!\!<\!\!0$.
Whence, $b_0^2\!-\!4c_0\!\!>\!\!0$ and so  $b_0^2\!-\!4c_0\!=\!d^2$ for some $d$. Now, $s\!=\!\frac{1}{2}(\!-\!b_0\!+\!d)$ is a root of $x^2\!+\!b_0x\!+\!c_0$ and so it is also a root of $\mathfrak{q}(x)$. Finally, for $r\!=\!u\!+\!\frac{v-u}{1+s^2}$ we have, obviously,
$u\!\!<\!\!r\!\!<\!\!v$ and $\mathfrak{p}(r)\!=\!0$.
  \end{proof}

So, by Theorems~\ref{th:ftaivt} ($\texttt{\textbf{FTA}}_{\textsf{RCF}}
\Longrightarrow\texttt{\textbf{IVT}}$) and~\ref{prop:ivtrcf} ($\texttt{\textbf{IVT}}
\Longrightarrow\texttt{\textbf{RCF}}$), the scheme $\texttt{\textbf{FTA}}_{\textsf{RCF}}$ implies that any odd-degree polynomial has a root.
So, one can prove, by an induction on the degree of the polynomials that, when $\texttt{\textbf{FTA}}_{\textsf{RCF}}$ holds, every polynomial can be written as the product of some linear and quadratic factors. Let us also note that $\texttt{\textbf{FTA}}_{\textsf{RCF}}
\Longrightarrow\texttt{\textbf{RCF}}$ could be proved directly as follows: for a given odd-degree polynomial $\mathfrak{p}(x)\!=\!x^{2n+1}\!+\!\sum_{i\leqslant 2n}a_ix^{i}$,  multiply it with $x$ to get an  even-degree polynomial
$x\cdot\mathfrak{p}(x)\!=\!x^{2n+2}
\!+\!\sum_{i=0}^{2n}a_ix^{i+1}$. By  $\texttt{\textbf{FTA}}_{\textsf{RCF}}$ there exist some $\{b_j\}_{j\leqslant n}$ and $\{c_j\}_{j\leqslant n}$ such that
$x^{2n+2}
\!+\!\sum_{i=0}^{2n}a_ix^{i+1}
\!=\!\prod_{j=0}^{n}(x^2\!+\!b_jx\!+\!c_j)$.
Since, $\prod_{j=0}^{n}c_j\!=\!0$, then for some $j$ we have  $c_j\!=\!0$. Assume that  $c_{0}\!=\!0$;  then we have   $x\cdot\mathfrak{p}(x)\!=\!x\cdot\big(x^{2n+1}\!+\!\sum_{i=0}^{2n}a_ix^{i}\big)
\!=\!x\cdot(x\!+\!b_{0})\cdot
\prod_{j=1}^{n}(x^2\!+\!b_jx\!+\!c_j)$. Thus,  we have the identity  $\mathfrak{p}(x)\!=\!(x\!+\!b_{0})\cdot
\prod_{j=1}^{n}(x^2\!+\!b_jx\!+\!c_j)$, and so $\mathfrak{p}(\!-\!b_0)\!=\!0$ holds.

\section{Conclusions and Open Problems}\label{sec:conc}
 Continuous Induction has been around in the literature over the last century (starting from \cite{chao} in 1919). We isolated three versions of it (Definitions~\ref{def:ind''},\ref{def:ind},\ref{def:ind'}, noting that all the other formats are equivalent to one of these three) and formalized them in first-order logic (Definitions~\ref{def:foind''},\ref{def:foind},\ref{def:foind'}) for the first time here. We showed that two of them are equivalent with each other (Theorem~\ref{th:indind'}) while the third one (and the oldest one) is not (Theorem~\ref{th:indind''}). Actually, those two strong versions are equivalent to the completeness of an ordered field, but the third one is equivalent to the Archimedean property (Theorem~\ref{th:arch}) and not with the completeness axiom. We noted that the first-order formulations of those two strong continuous induction schemes can completely axiomatize the real closed ordered fields (cf. Theorem~\ref{prop:infind}). For this theory we  collected some  axiomatizations (Definitions~\ref{def:sup},\ref{def:inf},\ref{def:cut},\ref{def:rcf},\ref{def:ivt})
which are equivalent with one another (Theorems~\ref{prop:supinf},\ref{prop:cutsup},\ref{prop:infind})
and added one; $\texttt{\textbf{FTA}}_{\textsf{RCF}}$
 a formalization of the Fundamental Theorem of Algebra (Definition~\ref{def:fta}).
The following diagram summarizes some of our  new and  old  results:
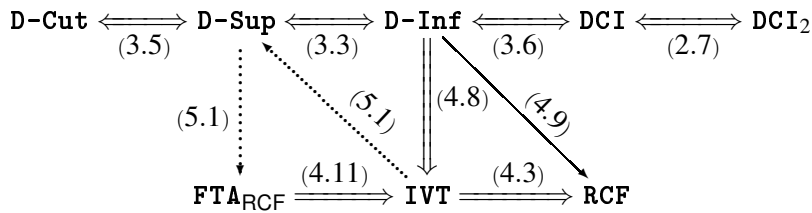
\begin{diagram}
\textbf{\texttt{D-Cut}}\; & \Leftarrow\hspace{-1em}\rImplies_{^{(\ref{prop:cutsup})}} & \textbf{\texttt{D-Sup}}\; & \Leftarrow\hspace{-1em}\rImplies_{^{(\ref{prop:supinf})}}   & \textbf{\texttt{D-Inf}}\; & \Leftarrow\hspace{-1em}\rImplies_{^{(\ref{prop:infind})}}  & \textbf{\texttt{DCI}}\; & \Leftarrow\hspace{-1em}\rImplies_{^{(\ref{th:indind'})}} & \texttt{\textbf{DCI}}_2 &  & \\
& & \dDotsto^{_{(\ref{prop:ivtcomp})}} & \luDotsto^{_{\qquad(\ref{prop:ivtcomp})}} & \dImplies_{^{(\ref{prop:infivt})}}  &
\rdTo^{_{\qquad(\ref{rem:infrcf})}}   &   &   &    &   \\
& & \texttt{\textbf{FTA}}_{\textsf{RCF}}   & \rImplies^{_{(\ref{th:ftaivt})}}  & \textbf{\texttt{IVT}} & \rImplies^{_{(\ref{prop:ivtrcf})}} & \textbf{\texttt{RCF}} &  &  &  &  \\
\end{diagram}
The completeness of $\!\texttt{\textbf{IVT}}$ ($\!+\texttt{\textbf{OF}}$) is proved in the Appendix (Theorem~\ref{prop:ivtcomp}); which, as was noted in Section~\ref{sec:rcf}, implies the equivalence of  $\texttt{\textbf{D-Inf}}$, $\texttt{\textbf{D-Sup}}$
 and  $\texttt{\textbf{FTA}}_{\textsf{RCF}}$ with  $\texttt{\textbf{IVT}}$, as well.

\begin{itemize}
\item[(\textbf{\texttt{1}})]  We leave open the existence of a nice and neat first-order proof of any implication of  the form $\texttt{\textbf{IVT}}\Longrightarrow\Theta$,  for
 $\Theta\!\in\!\{
 \texttt{\textbf{D-Inf}},
 \texttt{\textbf{D-Sup}},
 \texttt{\textbf{D-Cut}},
 \texttt{\textbf{DCI}},
 \texttt{\textbf{DCI}}_2,
 \texttt{\textbf{FTA}}_{\textsf{RCF}}\}$.
Let us note that our intent by iterating some classical theorems of Mathematical Analysis was to put emphasize on their first-order formalizability. Indeed, being formalizable in  logic is not a trivial matter. All our proofs (except the proof of Theorem~\ref{th:arch})  were first-order.
\item[(\textbf{\texttt{2}})]    The second open problem is   a nice and neat first-order proof of
$\texttt{\textbf{RCF}}\Longrightarrow
\texttt{\textbf{FTA}}_{\textsf{RCF}}$ (closing the gap in the  diagram).
As we noted earlier, there are some second-order proofs for this (e.g. the Proofs Three and Four for the Fundamental Theorem of Algebra in \cite{fta}). By G\"odel's Completeness Theorem (for first-order logic) there should exist such proofs; but what are they?
As a matter of fact, any such proof will be a beautiful   proof of the Fundamental Theorem of Algebra, which is completely real analytic (not referring to complex numbers) and it will be \emph{a first-order proof} of the theorem, for the first time.
\item[(\textbf{\texttt{3}})]    One research area which is wide open to explore, is the formalization of the other equivalent axiomatizations of the complete ordered fields  in first-order logic; and seeing whether any of them can completely axiomatize the theory of real closed fields (over $\texttt{\textbf{OF}}$).  Or more generally, does the theory $\texttt{\textbf{OF}}\!+\!\Theta$, where $\Theta$ is a first-order formalization of any of the 72 completeness axioms in \cite{deteisman}, completely axiomatize the theory of real closed fields?
This question is not easy, since for example neither $\texttt{\textbf{OF}}\!+\!\texttt{\textbf{Rolle}}$ nor
$\texttt{\textbf{OF}}\!+\!\texttt{\textbf{MVT}}$  is equivalent with the theory of real closed fields (see \cite{pelling,brcrpe}).
As we saw above, $\texttt{\textbf{OF}}\!+\!\Theta$ axiomatizes the real closed fields  if and only if $\texttt{\textbf{OF}}\!+\!\Theta\vdash\texttt{\textbf{IVT}}$.
Let us note that $\texttt{\textbf{FTA}}_{\textsf{RCF}}$ is not a first-order formalization of any of the completeness axioms in \cite{deteisman}.

\end{itemize}


\newpage

\section*{Appendix}\label{sec:app}
Here we give a proof of Tarski's Theorem on the completeness of the theory of real closed (ordered) fields; this is a slightly modified proof given in \cite{kk}.
\begin{theorem}[$\texttt{\textbf{IVT}}$ is
$\textrm{\footnotesize Complete}$]
\label{prop:ivtcomp}

\noindent
The theory \textup{$\texttt{\textbf{OF}}\!+\!\texttt{\textbf{IVT}}$}  is complete.
\end{theorem}
\begin{proof}
We will employ the method of  Quantifier Elimination; i.e., we prove that every formula in the language of ordered fields
 is equivalent with a quantifier-free formula with the same free variables, over the theory $\texttt{\textbf{OF}}\!+\!\texttt{\textbf{IVT}}$. Since this theory can decide (prove or refute) quantifier-free sentences, then the quantifier elimination theorem will show that $\texttt{\textbf{OF}}\!+\!\texttt{\textbf{IVT}}$ can decide every sentence in its language; i.e., either proves it or proves its negation.

\bigskip

 For that it suffices to show that every formula of the form $\exists x\,\theta(x)$, where $\theta$ is the conjunction of atomic or negated-atomic formulas, is equivalent with a quantifier-free formula (with the same free variables). To see this, suppose that every such formula is equivalent with a quantifier-free formula. Then by induction on the complexity of a  formula $\psi$ one can show that $\psi$ is equivalent with a quantifier-free formula: the case of atomic formulas and the propositional connectives $\{\neg,\wedge,\vee,\rightarrow\}$ are trivial, and the case of $\forall$ can be reduced to that of $\exists$ by $\forall x\,\psi(x)\equiv\neg\exists x\,\neg\psi(x)$. It remains to show the equivalence of $\exists x\,\psi(x)$ with  a quantifier-free formula, where $\psi$ is  a quantifier-free formula. Write $\psi$ in   disjunctive normal form $\bigvee\hspace{-1.5ex}\bigvee_i\theta_i$ where each $\theta_i$ is a conjunctions of some atomic and/or    negated-atomic formulas. So, $\exists x\,\psi(x)\equiv\exists x\,\bigvee\hspace{-1.5ex}\bigvee_i\theta_i\equiv
\bigvee\hspace{-1.5ex}\bigvee_i\exists x\,\theta_i$, and  by the induction hypothesis each of $\exists x\,\theta_i$ is equivalent to some quantifier-free formula; thus the whole formula is so.

\bigskip

Since every term in the language $\{+,0,-,\times,1\}$ with the free variable $x$ is a polynomial of $x$ whose coefficients are some $x$-free terms, then every atomic formula with $x$ is equivalent to $\mathfrak{p}(x)\!=\!0$ or $\mathfrak{q}(x)\!\!>\!\!0$ for some polynomials $\mathfrak{p},\mathfrak{q}$. Noting that the negation sign   $\neg$ can be eliminated by

$[\mathfrak{p}(x)\!\neq\!0]\equiv [\mathfrak{p}(x)\!\!>\!\!0\vee\mathfrak{p}(x)\!\!<\!\!0]$ and $[\mathfrak{q}(x)\!\not>\!0]\equiv[\mathfrak{q}(x)\!=\!0\vee \mathfrak{q}(x)\!\!<\!\!0]$,

\noindent
we can consider atomic formulas only. So, we consider  the formulas of the following form for some polynomials $\{\mathfrak{p}_i(x)\}_{i<\ell}$ and
$\{\mathfrak{q}_j(x)\}_{j<n}$:  $\exists x\,[\bigwedge\hspace{-1.5ex}\bigwedge_{i<\ell}
\mathfrak{p}_i(x)\!=\!0\;\wedge\;
\bigwedge\hspace{-1.5ex}\bigwedge_{j<n}
\mathfrak{q}_j(x)\!\!>\!\!0]$.
Finally, it suffices to consider the bounded formulas of the form $\exists x\!\in]a,b[:\psi(x)$ since every formula   $\exists x\,\psi(x)$ is equivalent with
$\exists x\!\!<\!\!\!-\!1\psi(x) \vee \psi(\!-\!1) \vee \exists x\!\in]\!-\!1,1[:\psi(x) \vee \psi(1) \vee \exists x\!\!>\!\!1\psi(x)$, and also we have

$\exists x\!\!<\!\!\!-\!1\psi(x)\equiv \exists y\!\in]0,1[:\psi(\!-\!y^{\!-\!1})$ and $\exists x\!\!>\!\!1\psi(x)\equiv \exists y\!\in]0,1[:\psi(y^{\!-\!1})$.

\noindent
Let us note that for any polynomials $\mathfrak{p}(x)\!=\!\sum_{i\leqslant m}a_ix^i$ and $\mathfrak{q}(x)\!=\!\sum_{j\leqslant l}b_jx^j$ and any $y\!\!>\!\!0$ we have
$$\begin{cases} \mathfrak{p}(y^{\!-\!1})\!=\!0 \leftrightarrow
\sum_{i\leqslant m}a_iy^{-i}\!=\!0\leftrightarrow\sum_{i\leqslant m}a_iy^{m-i}\!=\!0
\leftrightarrow\sum_{j\leqslant m}a_{m-j}y^j\!=\!0, & \textrm{and} \\ \mathfrak{q}(y^{\!-\!1})\!>\!0 \leftrightarrow
\sum_{j\leqslant l}b_jy^{-j}\!>\!0\leftrightarrow\sum_{j\leqslant l}b_jy^{l-j}\!>\!0
\leftrightarrow\sum_{i\leqslant l}a_{l-i}y^i\!>\!0; &  \end{cases}$$
and so any atomic formula $\theta(y^{-1})$ (and also $\theta(\!-\!y^{-1})$), for $y\!\!>\!\!0$, is equivalent with another atomic formula $\eta(y)$.

Whence, we show the equivalence of all the formulas in  the following form with a quantifier-free formula:
$$(\dagger)\qquad \varphi(a,b)\!=\!\exists x\!\in]a,b[:\bigwedge\hspace{-2.2ex}\bigwedge_{i<\ell}
\mathfrak{p}_i(x)\!=\!0\wedge
\bigwedge\hspace{-2.45ex}\bigwedge_{j<n}
\mathfrak{q}_j(x)\!\!>\!\!0$$
This will be proved by induction on the degree of a formula which we define as follows:

\medskip

\begin{itemize}
\item[---]
for a term $\mathfrak{p}(x)\!=\!\sum_{i\leqslant m}a_ix^i$ let $\deg_x \mathfrak{p}\!=\!m$;
\item[---]
for atomic formulas  $\mathfrak{p}(x)\!=\!0$ or $\mathfrak{q}(x)\!\!>\!\!0$,  let

\qquad {$\deg_x(\mathfrak{p}(x)\!=\!0)\!=\!\deg_x\mathfrak{p}$ and
$\deg_x(\mathfrak{q}(x)\!\!>\!\!0)\!=\!1\!+\!
\deg_x\mathfrak{q}$;}
\item[---]
finally the $\deg_x$ of a formula is the maximum of the $\deg_x$'s of its atomic sub-formulas.
\end{itemize}

\bigskip

\noindent What we prove is:

\medskip

\begin{center}
\noindent  $(\star)$ \quad
for any formula $\varphi(a,b)$ as ($\dagger$) above,
there are some formulas $\{\Phi_k(y),\Psi_k(y)\}_{k<m}$ 
such that $\texttt{\textbf{OF}}\!+\!\texttt{\textbf{IVT}}
\!+\!a\!\!<\!\!b\vdash
\varphi(a,b)\leftrightarrow
\bigvee\hspace{-1.5ex}\bigvee_{k<m}
[\Phi_k(a)\wedge\Psi_k(b)]$;

 and moreover the $\deg_y$ of $\{\Phi_k(y),\Psi_k(y)\}$'s are less than $\deg_x(\varphi(a,b))$.
\end{center}

\medskip

\noindent
This will be shown by induction on $\hbar\!=\!\deg_x\big(\varphi(a,b)\big)$; let us note that here $a,b$ are treated as (new) variables. If $\hbar\!=\!0$, then $x$ appears only superficially in $\varphi(a,b)$ and so it is equivalent with a quantifier-free formula. Now, suppose that we have the desired conclusion ($\star$) for all the formulas with $\deg_x$ less than $\hbar$.

\bigskip

\noindent ${\bf (1)}$ \quad
First, we consider the case of $\ell\!=\!0$; i.e., the formulas that are in
the following form:
$$(\dagger)\qquad \varphi(a,b)\!=\!\exists x\!\in]a,b[:
\bigwedge\hspace{-2.45ex}\bigwedge_{j<n}
\mathfrak{q}_j(x)\!\!>\!\!0.$$
By Lemma~\ref{lem:cont},  $\mathfrak{q}_j$'s are positive at a point in $]a,b[$ if and only if they are positive in some   sub-interval $]u,v[\,\subseteq\,]a,b[$. So, we have  $\varphi(a,b)\equiv {\sf F}(a,b)\vee\bigvee\hspace{-1.5ex}\bigvee_{i<n}{\sf G}_i(a,b)\vee\bigvee\hspace{-1.5ex}\bigvee_{i,j<n}{\sf H}_{i,j}(a,b)$,  where ${\sf F}$, $\{{\sf G}_i\}_{i<n}$ and $\{{\sf H}_{i,j}\}_{i,j<n}$ are  as follows:

\bigskip

\begin{itemize}
\item[]  ${\sf F}(a,b)\!=\!\forall x\!\in]a,b[:\bigwedge\hspace{-1.5ex}\bigwedge_{j<n}
\mathfrak{q}_j(x)\!\!>\!\!0$,
\item[] ${\sf G}_i(a,b)\!=\!\big[\exists u\!\in]a,b[:\mathfrak{q}_i(u)\!=\!0\wedge{\sf F}(a,u) \big]\bigvee\big[\exists v\!\in]a,b[:\mathfrak{q}_i(v)\!=\!0\wedge{\sf F}(v,b)\big]$, and
\item[] ${\sf H}_{i,j}(a,b)\!=\!\exists u,v\!\in]a,b[: u\!\!<\!\!v \wedge  \mathfrak{q}_i(u)\!=\!0 \wedge \mathfrak{q}_j(v)\!=\!0 \wedge {\sf F}(u,v)$.
\end{itemize}

\bigskip

\noindent
The formula ${\sf F}(a,b)$ is equivalent with  $\bigwedge\hspace{-1.5ex}\bigwedge_{j<n}
\mathfrak{q}_j(a)\!\geqslant\!0\wedge
\bigwedge\hspace{-1.5ex}\bigwedge_{j<n}\!\neg\exists x\!\in]a,b[:\mathfrak{q}_j(x)\!=\!0$, whose $\deg_x$ is less than $\hbar$, and so by the induction hypothesis ($\star$) is equivalent with the formula $\bigvee\hspace{-1.5ex}\bigvee_{k<m}
[\Phi_k(a)\wedge\Psi_k(b)]$ for some $x$-free formulas $\{\Phi_k(y),\Psi_k(y)\}_{k<m}$ whose $\deg_y$ are less than $\hbar$. So, for any $i\!\!<\!\!n$,   $G_i(a,b)$ is equivalent to the disjunction of the   formulas (over $k\!\!<\!\!m$):
$$\big(\Phi_k(a)\wedge\exists u\!\in]a,b[:[\mathfrak{q}_i(u)\!=\!0\wedge\Psi_k(u)]\big)
\bigvee\big(\exists v\!\in]a,b[:
[\mathfrak{q}_i(v)\!=\!0\wedge
\Phi_k(v)]\wedge\Psi_k(b)\big).$$
Since the $\deg_u$ of $\mathfrak{q}_i(u)\!=\!0\wedge\Psi_k(u)$ and the $\deg_v$ of $\mathfrak{q}_i(v)\!=\!0\wedge
\Phi_k(v)$ are less than $\hbar$ then  the induction hypothesis ($\star$) applies to all ${\sf G}_i$'s. Finally, for ${\sf H}_{i,j}(a,b)$ we note that it is equivalent with the disjunction (over $k<m$) of the following formulas:
$$\exists u\!\in]a,b[:\big(\mathfrak{q}_i(u)\!=\!0\wedge\Phi_k(u)
\wedge\exists v\!\in]u,b[:[\mathfrak{q}_j(v)\!=\!0\wedge\Psi_k(v)]\big).$$
Now, the formula $\mathfrak{q}_j(v)\!=\!0\wedge\Psi_k(v)$ has $\deg_v$ less than $\hbar$; so by the induction hypothesis ($\star$) there are formulas $\{\Theta_{j,k,\iota}(z),\Upsilon_{j,k,\iota}(z)\}_{\iota<l}$ with $\deg_z$ less than $\hbar$ such that  
$\exists v\!\in]u,b[:[\mathfrak{q}_j(v)\!=\!0\wedge\Psi_k(v)]$ is equivalent with $\bigvee\hspace{-1.5ex}\bigvee_{\iota<l}
[\Theta_{j,k,\iota}(u)\wedge\Upsilon_{j,k,\iota}(b)]$. Thus, ${\sf H}_{i,j}(a,b)$ is equivalent to the disjunction of  the formulas  $\exists u\!\in]a,b[:\big(\mathfrak{q}_i(u)\!=\!0\wedge\Phi_k(u)
\wedge\Theta_{j,k,\iota}(u)
\big)\wedge\Upsilon_{i,j,\iota}(b)$, for $k\!\!<\!\!m,\iota\!\!<\!\!l$,  to which the induction hypothesis ($\star$) apply, since the $\deg_u$ of the all the formulas $\mathfrak{q}_i(u)\!=\!0\wedge\Phi_k(u)
\wedge\Theta_{j,k,\iota}(u)$ are   $<\hbar$.

\bigskip

\noindent ${\bf (2)}$ \quad
Second, we consider the case of $\ell\!\!>\!\!0$; let us note that we can assume $\ell\!=\!1$ since we have $\bigwedge\hspace{-1.5ex}\bigwedge_{i<\ell}
\alpha_i\!=\!0 \iff \sum_{i<\ell}\alpha_i^2\!=\!0$. So, we may replace all the $\mathfrak{p}_i(x)$'s with a single polynomial $\sum_{i<\ell}\mathfrak{p}^2_i(x)$; but this will increase the $\deg_x$ of the resulted formula. There is another way of reducing $\ell$ (the number of polynomials $\mathfrak{p}_i$'s) without increasing the $\deg_x$ of the formula:
\begin{itemize}
\item[\underline{(I)}] For polynomials  $\mathfrak{p}(x)\!=\!\alpha x^{d}\!+\!\sum_{i<d}a_ix^i$ and $\mathfrak{q}(x)\!=\!\beta x^{e}\!+\!\sum_{j<e}b_jx^j$   assume that $d\!\!\geqslant\!\!e$. Put $\mathfrak{r}(x)\!=\!\mathfrak{q}(x)\!-\!\beta x^e$ and $\mathfrak{s}(x)\!=\!\beta\mathfrak{p}(x)-\alpha x^{d-e}\mathfrak{q}(x)$; then we have

\centerline{$\big[\mathfrak{p}(u)\!=\!
    \mathfrak{q}(u)\!=\!0\big]
    \iff\big[\beta\!=\!0\wedge\mathfrak{p}(u)\!=\!
\mathfrak{r}(u)\!=\!0\big]\vee\big[\beta\!\neq\!0\wedge
\mathfrak{s}(u)\!=\!\mathfrak{q}(u)\!=\!0\big]$.}

Continuing this way, at least one of the two polynomials will disappear and we will be left with at most  one polynomial; and the $\deg_x$ of the last formula will be non-greater than the $\deg_x$ of the first formula.
\end{itemize}
So, we can safely assume that $\ell\!=\!1$;
thus  $\varphi(a,b)\!=\!\exists x\!\in]a,b[:\mathfrak{p}(x)\!=\!0\wedge
\bigwedge\hspace{-1.5ex}\bigwedge_{j<n}
\mathfrak{q}_j(x)\!\!>\!\!0$ and  also  $\deg_x(\varphi(a,b))\!=\!\hbar$.
We can still transform the formula to an equivalent one in which we have that $\deg_x\mathfrak{q}_j\!\!<\!\!\deg_x\mathfrak{p}$ for all $j<n$:
\begin{itemize}
\item[\underline{(II)}]  If, say, $\deg_x\mathfrak{q}_1\!\!\geqslant\!
    \deg_x\mathfrak{p}$, then write $\mathfrak{p}(x)\!=\!\alpha x^{d}\!+\!\sum_{i<d}a_ix^i$ and $\mathfrak{q}_1(x)\!=\!\beta x^{e}\!+\!\sum_{j<e}b_jx^j$   with  $d\!\!\leqslant\!\!e$. Put $\mathfrak{r}(x)\!=\!\mathfrak{p}(x)\!-\!\alpha x^d$ and $\mathfrak{s}(x)\!=\alpha^2\!\mathfrak{q}_1(x)\!-\!
    \alpha\beta x^{e-d}\mathfrak{p}(x)$; then we have

{$\big[\mathfrak{p}(u)\!=\!0
\wedge\mathfrak{q}_1(u)\!\!>\!\!0\big] \iff  \big[\alpha\!=\!0\wedge\mathfrak{r}(u)\!=\!0\wedge
\mathfrak{q}_1(u)\!\!>\!\!0\big]
\vee\big[\alpha\!\neq\!0\wedge
\mathfrak{p}(u)\!=\!0\wedge \mathfrak{s}(u)\!\!>\!\!0\big]$.}
Continuing this way, either the equality ($\mathfrak{p}(x)\!=\!0$) will disappear (and so we will have the first case) or the degree of the inequality $\deg_x(\mathfrak{q}(x)\!\!>\!\!0)$ will be non-greater than the degree of the equality, $\deg_x(\mathfrak{p}(x)\!=\!0)$.
\end{itemize}
So, assume that in the formula   $$(\dagger)\quad\varphi(a,b)\!=\!\exists x\!\in]a,b[:\mathfrak{p}(x)\!=\!0\wedge
\bigwedge\hspace{-2.45ex}\bigwedge_{j<n}
\mathfrak{q}_j(x)\!\!>\!\!0$$ we have  that   $\deg_x(\varphi(a,b))\!=\!\hbar\!=\!\deg_x\mathfrak{p}$
 and also
$\bigwedge\hspace{-1.5ex}\bigwedge_{j<n}
\deg_x\mathfrak{q}_j\!\!<\!\!\hbar$.
 Now, the formula $\varphi(a,b)$ is equivalent to $\varphi_1(a,b)\vee\varphi_2(a,b)\vee\varphi_3(a,b)$ where
\begin{itemize}
\item[] $\varphi_1(a,b)\!=\!\!\exists x\!\in]a,b[:\mathfrak{p}(x)\!=\!0\wedge
    \mathfrak{p}'(x)\!=\!0\wedge
\bigwedge\hspace{-1.5ex}\bigwedge_{j<n}
\mathfrak{q}_j(x)\!\!>\!\!0$,
\item[] $\varphi_2(a,b)\!=\!\!\exists x\!\in]a,b[:\mathfrak{p}(x)\!=\!0\wedge
    \mathfrak{p}'(x)\!>\!0\wedge
\bigwedge\hspace{-1.5ex}\bigwedge_{j<n}
\mathfrak{q}_j(x)\!\!>\!\!0$, and
\item[]  $\varphi_3(a,b)\!=\!\!\exists x\!\in]a,b[:\mathfrak{p}(x)\!=\!0\wedge
    \mathfrak{p}'(x)\!<\!0\wedge
\bigwedge\hspace{-1.5ex}\bigwedge_{j<n}
\mathfrak{q}_j(x)\!\!>\!\!0$.
\end{itemize}
Since, $\deg_x\mathfrak{p}'\!\!<\!\!\deg_x\mathfrak{p}$ (see Definition~\ref{def:der}) then by applying \underline{(I)} to $\varphi_1$ we get an equivalent formula with $\deg_x$ less than $\hbar$ and then can apply the induction hypothesis ($\star$).
For $\varphi_2$ we note that, by  $\texttt{\textbf{IVT}}$ and Theorem~\ref{prop:ivt}, all the $\mathfrak{q}_j$'s and also $\mathfrak{p}'$ are strictly positive at some point in  $]a,b[$ in which $\mathfrak{p}$ vanishes, if and only if all the $\mathfrak{q}_j$'s and   $\mathfrak{p}'$ are strictly positive
on some open sub-interval  $]u,v[\,\subseteq\,]a,b[$ such that ($\mathfrak{p}$ is monotonically increasing and so) $\mathfrak{p}(u)\!\!<\!\!0\!\!<\!\!\mathfrak{p}(v)$.
Whence, $\varphi_2(a,b)$ is equivalent to the disjunction of the following formulas:
\begin{itemize}\itemindent=1em
\item[(i)] $\mathfrak{p}(a)\!\!<\!\!0\!\!<\!\!\mathfrak{p}(b)
    \wedge{\sf F}(a,b)$,
\item[(ii)] $\bigvee\hspace{-1.5ex}\bigvee_{i<n}
    \Big(\big[\mathfrak{p}(a)\!\!<\!\!0\wedge\exists u\!\in]a,b[:\big(\mathfrak{q}_i(u)\!=\!0\wedge
    \mathfrak{p}(u)\!\!>\!\!0\wedge{\sf F}(a,u)\big)\big]\bigvee$

   \hfill  $\big[\mathfrak{p}(b)\!\!>\!\!0\wedge\exists v\!\in]a,b[:\big(\mathfrak{q}_i(v)\!=\!0\wedge
    \mathfrak{p}(v)\!\!<\!\!0\wedge{\sf F}(v,b)\big)\big]\Big)$, and
\item[(iii)] $\bigvee\hspace{-1.5ex}\bigvee_{i,j<n}\Big(
    \exists u,v\in\!]a,b[:\big[\mathfrak{q}_i(u)\!=\!0\wedge
    \mathfrak{q}_j(v)\!=\!0\wedge
    \mathfrak{p}(u)\!\!<\!\!0\!\!<\!\!\mathfrak{p}(v)
    \wedge {\sf F}(u,v)\big]\Big)$.
\end{itemize}
The formula (i) has been treated before (it is equivalent to a formula with $\deg_x$ less than $\hbar$). The formulas (ii) and (iii) can be equivalently transformed to formulas with $\deg_x$ less than $\hbar$ by \underline{(I)} and \underline{(II)} above. So, the whole formula $\varphi_2$, and very similarly $\varphi_3$, can be written in equivalent forms in such a way that the induction hypothesis ($\star$) applies to them.
 \end{proof}

\end{document}